\documentclass[11pt,twoside,reqno]{amsart}

\usepackage{amssymb, latexsym}
\usepackage[italian,english]{babel}
\usepackage[latin1]{inputenc} 
\usepackage{paralist}
\usepackage[top=2.50cm, bottom=2.50cm, left=2.80cm, right=2.80cm]{geometry}
\usepackage{color}

\numberwithin{equation}{section}

\usepackage{xspace}
\usepackage[bookmarksnumbered,colorlinks]{hyperref}
\usepackage{graphics}
\def\bb#1\eb{\textcolor{blue}
{#1}} %
\def\br#1\er{\textcolor{red}
{#1}} %
\def\bv#1\ev{\textcolor{green}
{#1}} %
\def\bc#1\ec{\textcolor{cyan}
{#1}} %

\usepackage{graphics}

\def\Xint#1{\mathchoice
  {\XXint\displaystyle\textstyle{#1}}%
  {\XXint\textstyle\scriptstyle{#1}}%
  {\XXint\scriptstyle\scriptscriptstyle{#1}}%
  {\XXint\scriptscriptstyle\scriptscriptstyle{#1}}%
  \!\int}
\def\XXint#1#2#3{{\setbox0=\hbox{$#1{#2#3}{\int}$}
  \vcenter{\hbox{$#2#3$}}\kern-.5\wd0}}
\def\-int{\Xint -}

\newcommand{\R}{\mathbb{R}}
\newcommand{\N}{\mathbb{N}}

\DeclareMathOperator{\dive}{div}

\DeclareMathOperator{\X}{\mathbb{X}^{s}_{m}}
\DeclareMathOperator{\Y}{\mathbb{Y}^{s}_{m}}
\DeclareMathOperator{\Z}{\mathbb{Z}^{s}_{m}}
\DeclareMathOperator{\J}{\mathcal{J}}
\DeclareMathOperator{\h}{\mathbb{H}^{s}_{m}}
\DeclareMathOperator{\T}{\textup{Tr}}
\newcommand{\e}{\varepsilon}

\newtheorem{lem}{Lemma}[section]
\newtheorem{thm}{Theorem}[section]
\newtheorem{defn}{Definition}[section]

\begin{document}

\title[Periodic solutions for critical fractional equations]{Periodic solutions for critical fractional problems}
\author[V. Ambrosio]{Vincenzo Ambrosio}
\address{Vincenzo Ambrosio \hfill\break\indent
Dipartimento di Scienze Pure e Applicate (DiSPeA) \hfill\break\indent
Universit\`a degli Studi di Urbino 'Carlo Bo'
 \hfill\break\indent
Piazza della Repubblica, 13
61029 Urbino (Pesaro e Urbino, Italy)\hfill\break\indent}
\email{vincenzo.ambrosio2@unina.it}


\begin{abstract}
We deal with the existence of $2\pi$-periodic solutions to the following non-local critical problem
\begin{equation*}
\left\{
\begin{array}{ll}
[(-\Delta_{x}+m^{2})^{s}-m^{2s}]u=W(x)|u|^{2^{*}_{s}-2}u+ f(x, u) &\mbox{ in } (-\pi,\pi)^{N}   \\
u(x+2\pi e_{i})=u(x)    &\mbox{ for all } x \in \R^{N}, \quad i=1, \dots, N,
\end{array}
\right.
\end{equation*}
where $s\in (0,1)$, $N \geq 4s$, $m\geq 0$, $2^{*}_{s}=\frac{2N}{N-2s}$ is the fractional critical Sobolev exponent, $W(x)$ is a positive continuous function, and $f(x, u)$ is a superlinear $2\pi$-periodic (in $x$) continuous function with subcritical growth.\\ 
When $m>0$, the existence of a nonconstant periodic solution is obtained by applying the Linking Theorem, after transforming the above non-local problem into a degenerate elliptic problem in the half-cylinder $(-\pi,\pi)^{N}\times (0, \infty)$, with a nonlinear Neumann boundary condition, through a suitable variant of the extension method in periodic setting.
We also consider the case $m=0$ by using a careful procedure of limit. As far as we know, all these results are new.
\end{abstract}

\maketitle

\section{Introduction}

\noindent
In the past years, a great attention has been devoted to the study of  nonlinear elliptic equations involving the critical Sobolev exponent. For instance, motivated by some variational problems in geometry and physics where a lack of compactness occurs, such as the Yamabe's problem \cite{Aubin},  in the celebrated paper \cite{BN}, Brezis and Nirenberg  considered the following critical boundary value problem
\begin{equation}\label{BNP}
\left\{
\begin{array}{ll}
-\Delta_{x}u=u^{\frac{N+2}{N-2}}+ f(x, u) &\mbox{ in } \Omega  \\
u=0    &\mbox{ on } \partial \Omega \\
u>0  &\mbox{ in }  \Omega, 
\end{array}
\right.
\end{equation}
where $N\geq 3$, $\Omega\subset \R^{N}$ is a bounded open set, and $f$ is a lower-order perturbation of $u^{\frac{N+2}{N-2}}$. 
Under appropriate assumptions on the nonlinearity and on the dimension of the space, they proved the existence of a positive solution to (\ref{BNP}) via Mountain Pass theorem \cite{AR}.\\
Subsequently, many authors investigated existence and multiplicity of nontrivial solutions to (\ref{BNP}) or some variants, by using suitable variational methods; see for instance \cite{CFP, DevSol, GR, struwe, Tarantello}. For critical problems in $\R^{N}$, we also cite \cite{BC, CY, FF, Lions}.

In this paper, we focus our attention on the following critical fractional  problem with periodic boundary conditions
\begin{equation}\label{P}
\left\{
\begin{array}{ll}
[(-\Delta_{x}+m^{2})^{s}-m^{2s}]u=W(x)|u|^{2^{*}_{s}-2}u+ f(x, u) &\mbox{ in } (-\pi,\pi)^{N}   \\
u(x+2\pi e_{i})=u(x)    &\mbox{ for all } x \in \R^{N}, \quad i=1, \dots, N,
\end{array}
\right.
\end{equation}
where $s\in (0,1)$, $N \geq 4s$, $2^{*}_{s}:=\frac{2N}{N-2s}$, $m\geq 0$, $f:\R^{N+1}\rightarrow \R$ is a continuous function satisfying suitable growth assumptions, and $(e_{i})$ is the canonical basis in $\R^{N}$.\\
The non-local operator $(-\Delta_{x}+m^{2})^{s}$ is defined through the spectral decomposition, by using the powers of the eigenvalues of $-\Delta_{x}+m^{2}$ with periodic boundary conditions.\\
Let $u\in\mathcal{C}^{\infty}_{2\pi}(\R^{N})$, that is $u$ is infinitely differentiable in $\R^{N}$ and $2\pi$-periodic in each variable.
We know that $u$ can be expressed via Fourier series 
$$
u(x)=\sum_{k\in \mathbb{Z}^{N}} c_{k} \frac{e^{\imath k\cdot x}}{(2\pi)^{\frac{N}{2}}} \quad (x\in \R^{N}),
$$
where 
$$
c_{k}=\frac{1}{(2\pi)^{\frac{N}{2}}} \int_{(-\pi,\pi)^{N}} u(x)e^{- \imath k \cdot x}dx \quad (k\in \mathbb{Z}^{N})
$$
are the Fourier coefficients of $u$.\\
Then, the operator $(-\Delta_{x}+m^{2})^{s}$ is defined by setting 
\begin{equation}\label{nfrls}
(-\Delta_{x}+m^{2})^{s} \,u=\sum_{k\in \mathbb{Z}^{N}} c_{k} (|k|^{2}+m^{2})^{s} \, \frac{e^{\imath  k\cdot x}}{(2\pi)^{\frac{N}{2}}}.
\end{equation}
This operator can be extended by density for any $u$ belonging to the Hilbert space 
$$
\h=\Bigl\{u=\sum_{k\in \mathbb{Z}^{N}} c_{k}\frac{e^{\imath  k\cdot x}}{(2\pi)^{\frac{N}{2}}} \in L^{2}(-\pi,\pi)^{N}: \sum_{k\in \mathbb{Z}^{N}} (|k|^{2}+m^{2})^{s} \, |c_{k}|^{2}<\infty \Bigr\}
$$
endowed with the norm
$$
|u|_{\h}^{2}=\sum_{k\in \mathbb{Z}^{N}} (|k|^{2}+m^{2})^{s} |c_{k}|^{2}.
$$
When $m=0$, the operator in (\ref{nfrls}) arises in models with periodic boundary conditions; see for instance \cite{CC, KNV, RS1}. 
We recall that in $\R^{N}$, the study of $(-\Delta_{x}+m^{2})^{s}-m^{2s}$ is motivated by the fractional quantum mechanics; indeed, when $s=\frac{1}{2}$, the operator $\sqrt{-\Delta_{x}+m^{2}}-m$ corresponds to the Hamiltonian describing the motion of a free relativistic particle of mass $m$; see \cite{LL} for more details.
On the other hand, from a probabilistic point of view, the operator $(-\Delta_{x}+m^{2})^{s}-m^{2s}$ has an important role in the Stochastic Process Theory, because it is the infinitesimal generator of the so-called $2s$-stable relativistic process; see \cite{Ryz} and references therein. \\
Recently, the study of fractional and non-local operators of elliptic type has achieved an enormous popularity, thanks to the interesting theoretical structure of these operators, and in view of concrete applications such as phase transitions, flames propagation, quasi-geostrophic flows, population dynamics, American options in finance, conservation laws, crystal dislocation, minimal surfaces. 
The interested reader may consult \cite{DPV, MBRS} and references therein.\\
More in general, nonlinear equations involving fractional operators are currently actively studied.
Caffarelli et al. \cite{CSS} investigated the regularity of solutions for a fractional obstacle problem. 
Cabr\'e and Sol\`a Morales \cite{Cabsolmor}  analyzed the existence, uniqueness, symmetry and variational properties of layer solutions for $\sqrt{-\Delta_{x}}u=-G'(u)$ in $\R$, where $G$ is a double well potential with two absolute minima. Felmer et al. \cite{FQT} dealt with the existence, regularity and asymptotic behavior of positive solutions for a superlinear fractional Schr\"odinger equation in $\R^{N}$. 
Servadei and Valdinoci \cite{SV1} established, via min-max arguments, the existence of nontrivial solutions for equations driven by a non-local integrodifferential operator with homogeneous Dirichlet boundary conditions.
Stinga and Volzone \cite{StingaVol} proved the existence of noncostant least energy positive regular solutions for a fractional semilinear Neumann problem.\\
Specially, the existence and the multiplicity of solutions of critical fractional elliptic problems in bounded domains and in $\R^{N}$, have been widely investigated by many authors.\\
Firstly, we recall some fundamental results established in bounded domains. Servadei and Valdinoci \cite{SV3} (see also \cite{MBMaw, S1, S2}) obtained the existence of nontrivial solutions for a Brezis-Nirenberg type problem involving a non-local integrodifferential operator.
Barrios et al. \cite{BCPS} (see also \cite{tan}), studied the effect of lower order perturbations in the existence of positive solutions to a critical elliptic problem with spectral Laplacian.
Autuori et al. \cite{AFP} investigated the existence and the asymptotic behavior of non-negative solutions for a class of critical stationary Kirchhoff  problems involving a critical nonlinearity.\\ 
Secondly, we mention some results for critical problems in $\R^{N}$.
Shang et al. \cite{SZY} dealt with the existence and the multiplicity of ground states for a fractional Schr\"odinger equation by using the method of Nehari manifold and Ljusternik-Schnirelmann category theory; see also \cite{A6}.
Dipierro et al. \cite{DMPV} obtained, via Lyapunov-Schmidt technique,  some bifurcation results for a  fractional elliptic equation with critical exponent. Pucci and Saldi \cite{PucSal} established the existence and multiplicity of nontrivial non-negative entire (weak) solutions of a critical Kirchhoff eigenvalue problem, involving a general nonlocal integro-differential operator.
Teng \cite{teng} proved the existence of ground state solutions for a nonlinear fractional Schr\"odinger-Poisson system  in $\R^{3}$, via the monotonicity trick and a global compactness Lemma.

Motivated by the interest that the mathematical community has focused on fractional problems  involving the critical Sobolev exponent, the aim of this paper is to investigate the existence of nonconstant periodic solutions for the critical nonlinear problem (\ref{P}). We point out that, to the best of our knowledge, there are few existence results \cite{A2, A3, A4, A5, AMB5} for nonlocal equations with periodic boundary conditions, and all of them involve superlinear nonlinearities with subcritical growth. Therefore, the results that we present here, can be considered as the first existence results regarding critical fractional problems in periodic setting. \\ 
Now, we state our main assumptions. In order to find (weak) periodic solutions to (\ref{P}), we will assume that the nonlinearity $f:\R^{N+1}\rightarrow \R$ satisfies the following hypotheses:
\begin{compactenum}[($f1$)]
\item $f(x, t)$ is $2\pi$-periodic in $x \in \R^{N}$; that is $f(x+2\pi e_{i},t)=f(x, t)$ for any $x\in \R^{N}$, $t\in \R$; 
\item $f$ is continuous in $\R^{N+1}$; 
\item $f(x, t)=o(t)$ as $t \rightarrow 0$ uniformly in $x\in \R^{N}$;
\item there exist $2<p<2^{*}_{s}$ and $C>0$ such that
$$
|f(x, t)|\leq C(1+|t|^{p-1})
$$
for any $x\in \R^{N}$ and $t\in \R$;
\item there exists $\mu>2$ such that
$$
0<\mu F(x, t)\leq t f(x, t)
$$
for any $x\in \R^{N}$ and $t\in \R\setminus \{0\}$. Here $\displaystyle{F(x,t)=\int_{0}^{t} f(x,\tau) d\tau}$;
\item there exists a function $\bar{f}$ such that $f(x, t)\geq \bar{f}(t)$ a.e. for $x\in \mathcal{A}$ and $t\geq 0$, where $\mathcal{A}$ is some nonempty open set in $[-\pi, \pi]^{N}$ and the function $\displaystyle{\bar{F}(t)=\int_{0}^{t} \bar{f}(\tau) d\tau}$ satisfies 
$$
\lim_{\e\rightarrow 0} \e^{\min\left \{\frac{N+2s}{2},\frac{p(N-2s)}{2}\right\}} \int_{0}^{\e^{-1}} \bar{F}\left[\left(\frac{\e^{-1}}{1+t^{2}}\right)^{\frac{N-2s}{2}}\right] t^{N-1} dt=\infty.
$$
\end{compactenum}
Let us observe that if $\bar{F}(t)=|t|^{p}$, then this condition is obviously satisfied.
\smallskip

\noindent
Concerning the function $W: \R^{N}\rightarrow \R$, we suppose that  it satisfies the following properties:
\begin{compactenum}[($W1$)]
\item $W\in \mathcal{C}^{0}(\R^{N}, \R)$, $\displaystyle{\min_{x\in [-\pi, \pi]^{N}} W(x)>0}$ and $W(x+2\pi e_{i})=W(x)$ for all $x\in \R^{N}$; 
\item $\displaystyle{W(0)=\max_{x\in [-\pi, \pi]^{N}} W(x)}$ and $W(x)=W(0)+O(|x|^{2s})$ as $|x|\rightarrow 0$.
\end{compactenum}
We note that when $s=1$, the assumption $(W2)$ has been introduced by Escobar in \cite{Esc}.

\noindent
Our first main result can be stated as follows.
\begin{thm}\label{mthm1}
Let $m>0$ and $f:\R^{N+1} \rightarrow \R$ be a function satisfying the assumptions $(f1)$-$(f6)$. Assume that $W$ verifies $(W1)$-$(W2)$.
Then there exists a nonconstant solution $u\in \h$ to $(\ref{P})$. In particular, $u$ belongs to $\mathcal{C}^{0,\alpha}([-\pi,\pi]^{N})$ for some $\alpha \in (0,1)$.
\end{thm}

\noindent
The proof of the above theorem is obtained by applying critical point theory.
As customary in many fractional problems \cite{CabTan, CafSil, StingaTor}, to overcome the difficulty due to the presence of the involved nonlocal operator, we study problem $(\ref{P})$ by using an alternative formulation of the operator (\ref{nfrls}), which consists of realizing $(-\Delta_{x}+m^{2})^{s}$ as an operator that maps a Dirichlet boundary condition to a Neumann-type condition via an extension problem on the half-cylinder $(-\pi,\pi)^{N} \times (0,\infty)$ with periodic boundary conditions on $\partial (-\pi,\pi)^{N} \times [0,\infty)$.

\noindent
More precisely, as proved in \cite{A2, A3}, for any $u\in \h$ one considers the problem
\begin{equation*}
\left\{
\begin{array}{ll}
-\dive(\xi^{1-2s} \nabla v)+m^{2}\xi^{1-2s}v =0 &\mbox{ in }\mathcal{S}_{2\pi}:=(-\pi,\pi)^{N} \times (0,\infty)  \\
v_{| {\{x_{i}=0\}}}= v_{| {\{x_{i}=2\pi\}}} & \mbox{ on } \partial_{L}\mathcal{S}_{2\pi}:=\partial (-\pi,\pi)^{N} \times [0,\infty) \\
v(x,0)=u(x)  &\mbox{ on }  \partial^{0}\mathcal{S}_{2\pi}:=(-\pi,\pi)^{N} \times \{0\},
\end{array}
\right.
\end{equation*}
from where the operator $(-\Delta_{x}+m^{2})^{s}$ is obtained as 
$$
-\lim_{\xi\rightarrow 0} \xi^{1-2s} \frac{\partial v}{\partial \xi}(x,\xi) = \kappa_{s} (-\Delta_{x} + m^{2})^{s} u(x) 
$$
in weak sense and $\displaystyle{\kappa_{s}= 2^{1-2s} \frac{\Gamma(1-s)}{\Gamma(s)}}$; see \cite{CafSil}.

\noindent
Taking into account this fact, instead of (\ref{P}), we are led to consider the following problem
\begin{equation}\label{R}
\left\{
\begin{array}{ll}
-\dive(\xi^{1-2s} \nabla v)+m^{2}\xi^{1-2s}v =0 &\mbox{ in }\mathcal{S}_{2\pi}:=(-\pi,\pi)^{N} \times (0,\infty)  \\
v_{| {\{x_{i}=0\}}}= v_{| {\{x_{i}=2\pi\}}} & \mbox{ on } \partial_{L}\mathcal{S}_{2\pi}:=\partial (-\pi,\pi)^{N} \times [0,\infty) \\
\frac{\partial v}{\partial \nu^{1-2s}}=\kappa_{s} [m^{2s}v+W(x)|v|^{2^{*}_{s}-2}v+ f(x, v)]   &\mbox{ on }\partial^{0}\mathcal{S}_{2\pi}:=(-\pi,\pi)^{N} \times \{0\},
\end{array}
\right.
\end{equation}
where 
$$
\frac{\partial v}{\partial \nu^{1-2s}}:=-\lim_{\xi \rightarrow 0} \xi^{1-2s} \frac{\partial v}{\partial \xi}(x,\xi)
$$
is the conormal exterior derivative of $v$.\\
Then, it is clear that solutions to $(\ref{R})$ can be characterized as critical points of the following Euler-Lagrange functional 
 $$
 \mathcal{J}_{m}(v)=\frac{1}{2} \|v\|_{\X}^{2}-\frac{\kappa_{s} m^{2s}}{2}|\T(v)|_{L^{2}(-\pi,\pi)^{N}}^{2} -\frac{\kappa_{s}}{2^{*}_{s}}|W^{\frac{1}{2^{*}_{s}}}\T(v)|_{L^{2^{*}_{s}}(-\pi,\pi)^{N}}^{2^{*}_{s}}-\kappa_{s}\int_{\partial^{0}\mathcal{S}_{2\pi}} F(x,\T(v)) \,dx
 $$
defined on the space $\X$, which is the closure of the set of smooth and $T$-periodic (in $x$) functions in $\R^{N+1}_{+}$ with respect to the norm
$$
\|v\|_{\X}^{2}:=\iint_{\mathcal{S}_{2\pi}} \xi^{1-2s} (|\nabla v|^{2}+m^{2s} v^{2}) \, dx \,d\xi.
$$ 
Differently from the papers \cite{A2, A3, A4, AMB5} dealing with subcritical problems, the main difficulty in studying (\ref{R}), is the lack of compactness of the embedding $\h$ into the space $L^{2^{*}_{s}}(-\pi, \pi)^{N}$ (see for instance \cite{PP}), which does not permit to verify that the Palais-Smale condition holds in all energy range $\R$. 
To overcome this difficulty, we construct a suitable periodic cut-off function (see Lemma \ref{Linking2}), and we prove that, for any fixed $m>0$, the functional $\mathcal{J}_{m}$ has the geometric structure required by the Linking Theorem \cite{Rab}, and that $\mathcal{J}_{m}$ satisfies the Palais-Smale condition at every level  $c<c^{*}$, with $c^{*}$   related to the best constant of the embedding of the fractional Sobolev space $H^{s}(\R^{N})$ into $L^{2^{*}_{s}}(\R^{N})$ (see \cite{CT}). After that, we will also study the regularity of the critical points of $\mathcal{J}_{m}$, and we show that every solution to (\ref{P}) is H\"older continuous.
\medskip

\noindent
In the second part of the paper, we focus our attention on the existence of periodic solutions to (\ref{P}) in the case $m=0$. 
In order to accomplish our purpose, we first show that, for any $m>0$ sufficiently small, the critical levels  of the functionals $\mathcal{J}_{m}$ can be estimated from below and from above by two positive constants independent of $m$. 
Then, we exploit this information to pass to the limit as $m\rightarrow 0$ in (\ref{R}), and we deduce the existence of a solution to the problem
\begin{equation}\label{P'}
\left\{
\begin{array}{ll}
(-\Delta_{x})^{s} u=W(x)|u|^{2^{*}_{s}-2}u+f(x, u) &\mbox{ in } (-\pi,\pi)^{N}   \\
u(x+2\pi e_{i})=u(x)    &\mbox{ for all } x \in \R^{N}, \quad i=1, \dots, N.
\end{array}
\right.
\end{equation}
To prove that this solution is nonconstant, we take advantage of the lower bound for the critical level of $\mathcal{J}_{m}$, and we borrow some ideas used in the proof of the Palais-Smale compactness condition.
More precisely, we obtain the following result.
\begin{thm}\label{mthm2}
Under the same assumptions on $f$ and $W$ of Theorem \ref{mthm1}, 
the problem $(\ref{P'})$ admits a nonconstant solution $u\in \h \cap\, \mathcal{C}^{0,\alpha}([-\pi,\pi]^{N})$, for some $\alpha \in (0,1)$.
\end{thm}

\noindent
The paper is organized as follows. In Section $2$ we collect some preliminaries which we will use to study the problem $(\ref{P})$. In Section $3$ we recall that the problem (\ref{P}) can be realized in a local manner through the nonlinear problem (\ref{R}). 
In Section $4$ we show that, for any fixed $m>0$, the functional $\mathcal{J}_{m}$ satisfies the Linking hypotheses. In Section $5$ we study the regularity of solutions of problem $(\ref{P})$. In the last section, we prove that we can find a nontrivial solution to $(\ref{P'})$, by taking the limit as $m\rightarrow 0$ in $(\ref{R})$.

\section{Preliminaries}

\noindent
In this section we introduce some notations and facts which will be frequently used in the sequel of paper.
For more details we refer the reader to \cite{A2, A3}.\\
We denote the upper half-space in $\R^{N+1}$ by 
$$
\R^{N+1}_{+}=\{(x,\xi)\in \R^{N+1}: x\in \R^{N}, \xi>0 \}.
$$
Let $\mathcal{S}_{2\pi}=(-\pi,\pi)^{N}\times(0,\infty)$ be the half-cylinder in $\R^{N+1}_{+}$ with basis $\partial^{0}\mathcal{S}_{2\pi}=(-\pi,\pi)^{N}\times \{0\}$
and we denote by $\partial_{L}\mathcal{S}_{2\pi}=\partial (-\pi,\pi)^{N}\times [0,+\infty)$ its lateral boundary.

With $|u|_{r}$ we always denote the $L^{r}(-\pi,\pi)^{N}$-norm of $u \in L^{r}(-\pi,\pi)^{N}$, and we set 
$$
|u|_{r, W}^{r}=\int_{(-\pi, \pi)^{N}} W(x) |u(x)|^{r} dx.
$$

Let $s\in (0,1)$ and $m> 0$. Let $A\subset \R^{N}$ be a domain. We denote by $L^{2}(A\times \R_{+},\xi^{1-2s})$
the space of all measurable functions $v$ defined on $A\times \R_{+}$ such that
$$
\iint_{A\times \R_{+}} \xi^{1-2s} v^{2} dx\, d\xi<\infty.
$$
We say that  $v\in H^{1}_{m}(A\times \R_{+},\xi^{1-2s})$ if $v$ and its weak gradient $\nabla v$ belong to $L^{2}(A\times \R_{+},\xi^{1-2s})$.
The norm of $v$ in $H^{1}_{m}(A\times \R_{+},\xi^{1-2s})$ is given by
$$
\iint_{A\times \R_{+}} \xi^{1-2s} (|\nabla v|^{2}+m^{2}v^{2}) \,dx\,d\xi<\infty.
$$
It is clear that $H^{1}_{m}(A\times \R_{+},\xi^{1-2s})$ is a Hilbert space with the inner product
$$
\iint_{A\times \R_{+}} \xi^{1-2s} (\nabla v \nabla z+m^{2}v z) \,dx\,d\xi.
$$
When $m=1$, we set  $H^{1}(A\times \R_{+},\xi^{1-2s})=H^{1}_{1}(A\times \R_{+},\xi^{1-2s})$.

We denote by $\mathcal{C}^{\infty}_{2\pi}(\R^{N})$ the space of functions 
$u\in \mathcal{C}^{\infty}(\R^{N})$ such that $u$ is $2\pi$-periodic in each variable, that is
$$
u(x+2\pi e_{i})=u(x) \mbox{ for all } x\in \R^{N}, i=1, \dots, N. 
$$
Let $u\in  \mathcal{C}^{\infty}_{2\pi}(\R^{N})$. Then we know that 
$$
u(x)=\sum_{k\in \mathbb{Z}^{N}} c_{k} \frac{e^{\imath  k\cdot x}}{(2\pi)^{\frac{N}{2}}} \quad \mbox{ for all } x\in \R^{N},
$$
where 
$$
c_{k}=\frac{1}{(2\pi)^{\frac{N}{2}}} \int_{(-\pi,\pi)^{N}} u(x)e^{-\imath k \cdot x}dx \quad (k\in \mathbb{Z}^{N})
$$ 
are the Fourier coefficients of $u$. We define the fractional Sobolev space $\h$ as the closure of 
$\mathcal{C}^{\infty}_{2\pi}(\R^{N})$ under the norm 
\begin{equation*}\label{h12norm}
|u|^{2}_{\h}:= \sum_{k\in \mathbb{Z}^{N}} (|k|^{2}+m^{2})^{s} \, |c_{k}|^{2}. 
\end{equation*}
When $m=1$, we set $\mathbb{H}^{\rm s}=\mathbb{H}^{\rm s}_{1}$ and $|\cdot |_{\mathbb{H}^{\rm s}}=|\cdot|_{\mathbb{H}^{\rm s}_{1}}$.
We also use the notation $\dot{\mathbb{H}}^{s}$ to denote the closure of $\mathcal{C}^{\infty}_{2\pi}(\R^{N})$ with respect to the following Gagliardo semi-norm 
$$
[u]=\sqrt{\sum_{k\in \mathbb{Z}^{N}} |k|^{2s} |c_{k}|^{2}}.
$$

Now, let us introduce the space of periodic function with respect to the $x$-component, that is 
\begin{align*}
\mathcal{C}_{2\pi}^{\infty}(\overline{\R^{N+1}_{+}})=\Bigl\{ & v\in \mathcal{C}^{\infty}(\overline{\R^{N+1}_{+}}): v(x+2\pi e_{i},y)=v(x,y) \\
&\mbox{ for every } (x,y)\in \overline{\R_{+}^{N+1}}, i=1, \dots, N \Bigr\}.
\end{align*}
We denote by $\dot{\mathbb{X}}^{s}$ the completion of $\mathcal{C}_{2\pi}^{\infty}(\overline{\R^{N+1}_{+}})$ with respect to the norm
$$
\|\nabla v\|^{2}_{L^{2}(\mathcal{S}_{2\pi}, \xi^{1-2s})}=\iint_{\mathcal{S}_{2\pi}} \xi^{1-2s}|\nabla v| dx \,d\xi.
$$
Finally, we define the functional space $\X$ as the completion of $\mathcal{C}_{2\pi}^{\infty}(\overline{\R^{N+1}_{+}})$ under the $H^{1}_{\rm m}(\mathcal{S}_{2\pi},\xi^{1-2s})$ norm 
\begin{equation*}
\|v\|^{2}_{\X}:=\iint_{\mathcal{S}_{2\pi}} \xi^{1-2s} (|\nabla v|^{2}+m^{2}v^{2}) \, dx\,d\xi.
\end{equation*} 

We recall that it is possible to define a trace operator from $\X$ to $\h$.
\begin{thm}\cite{A2, A3}\label{tracethm}
There exists a bounded linear operator $\textup{Tr} : \X \rightarrow \h$  such that :
\begin{itemize}
\item[(i)] $\textup{Tr}(v)=v|_{\partial^{0} \mathcal{S}_{2\pi}}$ for all $v\in \mathcal{C}_{2\pi}^{\infty}(\overline{\R^{N+1}_{+}}) \cap \X$;
\item[(ii)] It holds 
$$
\sqrt{\kappa_{s}}  |\textup{Tr}(v)|_{\mathbb{H}^{s}_{m}}\leq \|v\|_{\X} \mbox{ for every } v\in \X; 
$$ 
\item[(iii)] $\textup{Tr}$ is surjective.
\end{itemize}
\end{thm} 

We also recall the following fundamental embeddings.
\begin{thm}\cite{A2, A3}\label{thm2}
Let $N> 2s$. Then  $\textup{Tr}(\mathbb{X}^{s}_{m})$ is continuously embedded in $L^{q}(-\pi,\pi)^{N}$ for any  $1\leq q \leq 2^{*}_{s}$.  Moreover,  $\textup{Tr}(\mathbb{X}^{s}_{m})$ is compactly embedded in $L^{q}(-\pi,\pi)^{N}$  for any  $1\leq q < 2^{*}_{s}$. 
\end{thm}

\section{Extension problem}

\noindent
In this section we show that the study $(\ref{P})$ is equivalent to investigate the solutions of a problem in a half-cylinder with a nonlinear Neumann boundary condition.\\
More precisely, the following result holds.

\begin{thm}\cite{A2, A3}
Let $u\in \mathbb{H}^{s}_{m}$. Then there exists a unique $v\in \mathbb{X}^{s}_{m}$ such that
\begin{equation}\label{extPu}
\left\{
\begin{array}{ll}
-\dive(\xi^{1-2s} \nabla v)+m^{2}\xi^{1-2s}v =0 &\mbox{ in }\mathcal{S}_{2\pi}  \\
v_{| {\{x_{i}=0\}}}= v_{| {\{x_{i}=2\pi\}}} & \mbox{ on } \partial_{L}\mathcal{S}_{2\pi} \\
v(\cdot,0)=u  &\mbox{ on } \partial^{0}\mathcal{S}_{2\pi}
\end{array}
\right.
\end{equation}
and
\begin{align}\label{conormal}
-\lim_{\xi \rightarrow 0} \xi^{1-2s}\frac{\partial v}{\partial \xi}(x,\xi)=\kappa_{s} (-\Delta_{x}+m^{2})^{s}u(x) \mbox{ in } \mathbb{H}^{-s}_{m},
\end{align}
where 
$$
\mathbb{H}^{-s}_{m}=\left \{u=\sum_{k\in \mathbb{Z}^{N}} c_{k} \frac{e^{\imath k\cdot x}}{(2\pi)^{\frac{N}{2}}}: \sum_{k\in \mathbb{Z}^{N}}\frac{|c_{k}|^{2}}{(|k|^{2}+m^{2})^{s}}< \infty \right\}
$$ 
is the dual of $\mathbb{H}^{s}_{m}$.
\end{thm}

\noindent
Hence, for any given $u\in \mathbb{H}^{s}_{m}$ we can find a unique function $v=\textup{Ext}(u)\in \X$, which will be called the periodic extension of $u$, such that
\begin{compactenum}[$(E1)$]
\item $v$ is smooth for $y>0$, $2\pi$-periodic in $x$ and $v$ solves (\ref{extPu});
\item $\|v\|_{\X}\leq \|z\|_{\mathbb{X}^{s}_{m}}$ for any $z\in \X$ such that $\textup{Tr}(z)=u$;
\item $\|v\|_{\X}=\sqrt{\kappa_{s}} |u|_{\h}$;
\item We have
$$
-\lim_{\xi\rightarrow 0}\xi^{1-2s} \frac{\partial v}{\partial \xi}(x,\xi)= \kappa_{s} (-\Delta_{x}+m^{2})^{s}u(x) \mbox{ in } \mathbb{H}^{-s}_{m}.
$$
\end{compactenum}


Taking into account the previous results, we can reformulate nonlocal periodic problems in a local way. 

Let $g \in \mathbb{H}^{-s}_{m}$ and consider the following two problems:
\begin{equation}\label{P*}
\left\{
\begin{array}{ll}
(-\Delta_{x}+m^{2})^{s}u=g &\mbox{ in } (-\pi,\pi)^{N}  \\
u(x+2\pi e_{i})=u(x) & \mbox{ for } x\in \R^{N}
\end{array}
\right.
\end{equation}
and
\begin{equation}\label{P**}
\left\{
\begin{array}{ll}
-\dive(\xi^{1-2s} \nabla v)+m^{2}\xi^{1-2s}v =0 &\mbox{ in }\mathcal{S}_{2\pi}  \\
v_{| {\{x_{i}=0\}}}= v_{| {\{x_{i}=T\}}} & \mbox{ on } \partial_{L}\mathcal{S}_{2\pi} \\
\frac{\partial v}{\partial \nu^{1-2s}}=\kappa_{s} g(x)  &\mbox{ on } \partial^{0}\mathcal{S}_{2\pi}.
\end{array}
\right.
\end{equation}

Then, we can define the concept of solution to the nonlocal problem $(\ref{P*})$ in terms of solutions to $(\ref{P**})$ as explained below.

\begin{defn}
We say that $v\in \X$ is a solution to $(\ref{P**})$, if for any $\phi\in \X$ it holds
$$
\iint_{\mathcal{S}_{2\pi}} \xi^{1-2s} (\nabla v \nabla \phi + m^{2} v \phi ) \, dx\,d\xi=\kappa_{s}\langle g, \textup{Tr}(\phi)\rangle_{\mathbb{H}^{-s}_{m}, \h}, 
$$
where $\langle \cdot, \cdot\rangle_{\mathbb{H}^{-s}_{m}, \h}$ is the duality pairing between $\h$ and $\mathbb{H}^{-s}_{m}$.  
\end{defn}
\begin{defn}
We say that $u\in \h$ is a weak solution to $(\ref{P*})$ if $u=\textup{Tr}(v)$
and $v$ is a weak solution to $(\ref{P**})$.
\end{defn}


Finally, we recall the following useful result:
\begin{thm}\cite{A3}\label{thm6}
\begin{equation}\label{eqY}
\|v\|_{\X}^{2} - \kappa_{s} m^{2s} |\textup{Tr}(v)|_{2}^{2} =0 \Leftrightarrow v(x,y)=C\, \theta(my) \mbox{ for some } C\in \R. 
\end{equation}
Here $\theta(\xi)=\frac{2}{\Gamma(s)} (\xi/2)^{s} K_{s}(\xi)$, and $K_{s}$ is the modified Bessel function of the second type with order $s$; see \cite{Erd}.
\end{thm}

\section{Periodic solutions in the cylinder $\mathcal{S}_{2\pi}$}

\noindent
In this section we prove the existence of a solution to $(\ref{P})$. As shown in the previous section, we know that the study of $(\ref{P})$ is equivalent to investigate the existence of weak solutions to
\begin{equation}
\left\{
\begin{array}{ll}
-\dive(\xi^{1-2s} \nabla v)+m^{2}\xi^{1-2s}v =0 &\mbox{ in }\mathcal{S}_{2\pi}:=(-\pi,\pi)^{N} \times (0,\infty)  \\
v_{| {\{x_{i}=0\}}}= v_{| {\{x_{i}=2\pi\}}} & \mbox{ on } \partial_{L}\mathcal{S}_{2\pi}:=\partial (-\pi,\pi)^{N} \times [0,\infty) \\
\frac{\partial v}{\partial \nu^{1-2s}}=\kappa_{s} [m^{2s}v+W(x)|v|^{2^{*}_{s}-2}v+ f(x,v)]   &\mbox{ on } \partial^{0}\mathcal{S}_{2\pi}:=(-\pi,\pi)^{N} \times \{0\}. 
\end{array}
\right.
\end{equation}
For simplicity, let us assume $\kappa_{s}=1$.\\
Then, we will look for the critical points of
$$
\mathcal{J}_{m}(v)=\frac{1}{2}\|v\|^{2}_{\X} -\frac{m^{2s}}{2}|\T(v)|_{2}^{2}- \frac{1}{2^{*}_{s}}|\T(v)|_{2^{*}_{s}, W}^{2^{*}_{s}}- \int_{\partial^{0}\mathcal{S}_{2\pi}} F(x,\T(v)) dx
$$
defined for $v\in\X$.

More precisely, we will prove that $\J_{m}$ satisfies the assumptions of the Linking Theorem \cite{Rab}:
\begin{thm}\label{LinkingThm}
Let $(X, \|\cdot\|)$ be a real Banach space with $X=Y\bigoplus Z$, where $Y$ is finite dimensional. 
Let $R>r>0$ and $z\in Z$ such that $\|z\|=r$. \\
Define the following sets
\begin{align*}
&M=\{v=y+t z: y\in Y, \|v\|\leq R \mbox{ and } t\geq 0\},\\
&\partial M=\{v=y+t z: y\in Y, \|v\|= R, t\geq 0 \mbox{ or } \|v\|\leq  R, t=0 \},\\
&N=\{v\in Z: \|v\|=r\}.
\end{align*}
Let $\J\in \mathcal{C}^{1}(X,\R)$ be such that
$$
b:=\inf_{v\in N} \J(v)>a:=\max_{v\in \partial M} \J(v).
$$
If $\J$ satisfies the Palais-Smale condition at the level $c$, which is defined by setting
$$
c:=\inf_{\gamma \in \Gamma} \max_{v\in M} \J(\gamma(v))
$$
where 
$$
\Gamma:=\{\gamma \in \mathcal{C}(M, X): \gamma=Id \mbox{ on } \partial M\},
$$
then $c$ is a critical point of $\J$.
\end{thm}

By using the assumptions on $f$, it is easy to prove that  $\mathcal{J}_{m}$ is well defined on $\X$ and $\mathcal{J}_{m}\in \mathcal{C}^{1}(\X,\R)$. Moreover, by using the trace inequality, we notice that the quadratic part of $\mathcal{J}_{m}$ is nonnegative, that is
\begin{equation}\label{partequad}
\|v\|^{2}_{\X} -m^{2s}|\T(v)|_{2}^{2}\geq 0. 
\end{equation}

Let us note (see \cite{A2, A3}) that
$$
\X=< \theta(m\xi)> \oplus \Bigl\{ v\in \X: \int_{\partial^{0}\mathcal{S}_{2\pi}} \T(v) \,dx=0\Bigr \}=:\Y \oplus \Z,
$$
where $\dim \Y<\infty$ and $\Z$ is the orthogonal complement of $\Y$ with respect to the inner product in $\X$.
In order to prove that $\mathcal{J}_{m}$ verifies the Linking hypotheses, we prove the following lemmas.

\begin{lem}\label{Linking1}
There exist $\rho>0$ and $\eta>0$ such that
$\mathcal{J}_{m}(v)\geq \rho \mbox{ for } v\in \Z: \|v\|_{\X}=\eta$.
\end{lem}
\begin{proof}
Firstly we show that there exists a constant $C_{m}>0$ such that
\begin{align}\label{eqnorm}
\|v\|_{\X}^{2}-m^{2s}|\T(v)|_{2}^{2}\geq C_{m} \|v\|_{\X}^{2}
\end{align}
for any $v\in \Z$.
Assume by contradiction that there exists a sequence $(v_{j})\subset \Z$ such that
$$
\|v_{j}\|_{\X}^{2}-m^{2s}|\T(v_{j})|_{2}^{2}< \frac{1}{j} \|v_{j}\|_{\X}^{2}.
$$
Let $z_{j}=v_{j}/\|v_{j}\|_{\X}$. Then $\|z_{j}\|_{\X}=1$, so we can assume that $z_{j}\rightharpoonup z$ in $\X$ and $\T(z_{j}) \rightarrow \T(z)$ in $L^{2}(-\pi,\pi)^{N}$ for some $z\in \Z$ ($\Z$ is weakly closed).

Hence, for any $j\in \N$
$$
1-m^{2s}|\T(z_{j})|_{2}^{2}<\frac{1}{j},
$$
so  we get $|\T(z_{j})|^{2}_{2}\rightarrow \frac{1}{m^{2s}}$ that is $|\T(z)|_{2}=\frac{1}{m^{s}}$.

On the other hand
\begin{align*}
0&\leq \|z\|_{\X}^{2}-m^{2s}|\T(z)|_{2}^{2}\\
&\leq \liminf_{j\rightarrow \infty}\|z_{j}\|_{\X}^{2}-m^{2s}|\T(z_{j})|_{2}^{2}=0
\end{align*}
implies that $z=c\, \theta(m\xi)$ by (\ref{eqY}). But $z\in \Z$ so $c=0$ and this is a contradiction because of $|\T(z)|_{2}=\frac{1}{m^{s}}>0$. This completes the proof of (\ref{eqnorm}).

It follows from $(f3)$ and $(f4)$ that for every $\e>0$ there exists $C_{\e}>0$ such that 
$$
\left|F(x, t)\right|\leq \e t^{2}+C_{\e}|t|^{p} \mbox{ for all } t\in \R.
$$
By applying Sobolev inequality (see Theorem \ref{thm2}) we can see that 
$$
\left|\int_{\partial^{0}\mathcal{S}_{2\pi}} F(x, \T(v)) \, dx\right|\leq \left(\frac{\e}{m^{2s}} \|v\|_{\X}^{2}+C'_{*}C_{\e}\|v\|_{\X}^{p}\right).
$$
This  and (\ref{eqnorm}) give
\begin{align*}
\mathcal{J}_{m}(v)&\geq C_{m}\|v\|_{\X}^{2}-\frac{W(0)}{2^{*}_{s}}|\T(v)|_{2^{*}_{s}}^{2^{*}_{s}} - \int_{\partial^{0}\mathcal{S}_{2\pi}} F(x, \T(v)) \, dx\\
&\geq \Bigl(C_{m}-\frac{\varepsilon}{m^{2s}}\Bigr)\|v\|_{\X}^{2}-C''_{*}\|v\|^{2^{*}_{s}}_{\X}-C'_{*}C_{\e} \|v\|_{\X}^{p}
\end{align*}
for any $v\in \Z$.
Choosing $\varepsilon$ sufficiently small, there exist $\rho>0$ and $\eta>0$ such that
$$
\mathcal{J}_{m}(v)\geq \rho \mbox{ for all  } v\in \Z: \|v\|_{\X}=\eta.
$$
\end{proof}

Now, we collect some preliminary lemmas which we will used later. 
First, we have the following result whose proof can be obtained following \cite{BCPS}.
\begin{lem}\label{lemma1}
Let $\eta$ be a cut-off function such that $\eta(t)=1$ for $t\in [0, \frac{1}{2}]$, $\eta(t)=0$ for $t\geq 1$, and we consider 
$\phi(x, \xi)=\eta\left(\frac{\sqrt{|x|^{2}+\xi^{2}}}{r}\right)$, where $r>0$ is such that $\overline{{B}^{+}_{r}}=\{(x, \xi)\in \R^{N+1}_{+}: \sqrt{|x|^{2}+\xi^{2}}\leq r\} \subset \overline{\mathcal{S}}_{2\pi}$.
Let 
$\varphi_{\e}(x, \xi)=\frac{1}{(2\pi)^{N}} \phi(x, \xi) \psi_{\e}(x, \xi)$ where $\psi_{\e}(\cdot, \xi)=P^{s}_{\xi}(\cdot, \xi)*w_{\e}$, 
$P^{s}_{\xi}(x, \xi)=C_{N, s} \frac{\xi^{2-2s}}{(|x|^{2}+\xi^{2})^{\frac{N+2-2s}{2}}}$ is the Poisson kernel \cite{CafSil}, and $w_{\e}(x)=\frac{\e^{\frac{N-2s}{2}}}{(\e^{2}+|x|^{2})^{\frac{N-2s}{2}}}$.\\
Then the following estimates hold:
\begin{compactenum}[(i)]
\item $\iint_{\mathcal{S}_{2\pi}} \xi^{1-2s} |\nabla \varphi_{\e}|^{2} dx d\xi\leq S_{*}^{\frac{N}{2s}}+O(\e^{N-2s})$,  where $S_{*}$ is the best Sobolev constant of the embedding $H^{s}(\R^{N})\subset L^{2^{*}_{s}}(\R^{N})$ (see \cite{CT}).
\item $|\varphi_{\e}(\cdot, 0)|_{2^{*}_{s}}^{2^{*}_{s}}=S_{*}^{\frac{N}{2s}}+O(\e^{N})$.
\item 
\begin{equation*}
|\varphi_{\e}(\cdot, 0)|_{2}^{2}=
\left\{
\begin{array}{ll}
K_{1} \e^{2s}+O(\e^{N-2s}) &\mbox{ if } N>4s \\
K_{1} \e^{2s}|\log \varepsilon|+O(\e^{2s}) &\mbox{ if } N=4s.
\end{array}
\right.
\end{equation*}
\item $|\varphi_{\e}(\cdot, 0)|_{1}\leq K_{2} \e^{\frac{N-2s}{2}}$.
\item $|\varphi_{\e}(\cdot, 0)|_{2^{*}_{s}-1}^{2^{*}_{s}-1}\leq K_{3} \e^{\frac{N-2s}{2}}$.
\item 
\begin{equation*}
|\varphi_{\e}(\cdot, 0)|_{q}^{q}=
\left\{
\begin{array}{ll}
K_{4} \e^{N-\frac{(N-2s)}{2} q} &\mbox{ if } q>\frac{N}{N-2s} \\
K_{4}  \e^{q\frac{(N-2s)}{2}} &\mbox{ if } q<\frac{N}{N-2s}.
\end{array}
\right.
\end{equation*}
\end{compactenum}
\end{lem}

Arguing as in the proof of \cite{CFP}, we are able to prove that
\begin{lem}\label{lemma2}
Let $v=y+t z_{\e}\in Q_{\e}:=\{y+t z_{\e}: y\in \Y, t\geq 0\}$, where $z_{\e}=[\varphi_{\e}-(\varphi_{\e})_{\Pi}]\theta(m\xi)$ and $(\varphi_{\e})_{\Pi}:= \frac{1}{(2\pi)^{N}}\int_{(-\pi, \pi)^{N}} \varphi_{\e}(x, \xi)\, dx$. Then, for any $\e>0$ we have
\begin{align*}
\int_{\partial^{0}\mathcal{S}_{2\pi}} &W(x)|\T(v)|^{2^{*}_{s}} dx \\
&\geq \int_{\partial^{0}\mathcal{S}_{2\pi}}W(x) |t \T(z_{\e})|^{2^{*}_{s}} dx+\frac{1}{2} \int_{\partial^{0}\mathcal{S}_{2\pi}} W(x)|\T(y)|^{2^{*}_{s}} dx-K_{5} t^{2^{*}_{s}} \e^{\frac{N(N-2s)}{N+2s}}.
\end{align*}
\end{lem}
\begin{proof}
Firstly, we observe that by $(iv)$ of Lemma \ref{lemma1} we have
\begin{equation}\label{phiinfA}
|(\T(\varphi_{\e}))_{\Pi}|_{2}^{2}\leq C_{1} |\T(\varphi_{\e})|_{1}^{2}\leq C_{2}\e^{N-2s}
\end{equation}
which implies that 
\begin{equation}\label{phiinf}
|(\T(\varphi_{\e}))_{\Pi}|_{\infty}\leq C_{3}\e^{\frac{N-2s}{2}}.
\end{equation}
Recalling the following identity
\begin{align}\label{identity}
|u|^{2^{*}_{s}}_{2^{*}_{s}, W}=2^{*}_{s} \int_{(-\pi, \pi)^{N}} W(x) dx \int_{0}^{u} |\tau|^{{2^{*}_{s}}-2} \tau d\tau,
\end{align}
and by using $(W2)$, (\ref{phiinfA}), (\ref{phiinf})  and $(v)$ of Lemma \ref{lemma1}, we have
\begin{align}\label{CY2star}
&\left|\int_{\partial^{0}\mathcal{S}_{2\pi}} W(x)(|\T(z_{\e})|^{2^{*}_{s}}-|\T(\varphi_{\e})|^{2^{*}_{s}}) dx \right| \nonumber \\
&=\left|2^{*}_{s}\int_{0}^{1} d\tau \int_{\partial^{0}\mathcal{S}_{2\pi}} W(x)\left[|T(\varphi_{\e})-\tau (\T(\varphi_{\e}))_{\Pi}|^{2^{*}_{s}-2}(T(\varphi_{\e})-\tau (\T(\varphi_{\e}))_{\Pi})\right] (\T(\varphi_{\e}))_{\Pi} dx \right| \nonumber \\
&\leq C_{4}W(0)\left(|\T(\varphi_{\e})|_{2^{*}_{s}-1}^{2^{*}_{s}-1} |(\T(\varphi_{\e}))_{\Pi}|_{\infty}+|(\T(\varphi_{\e}))_{\Pi}|_{2}^{2^{*}_{s}}\right) \nonumber \\
&\leq C_{5}\e^{N-2s}.
\end{align}
Moreover
\begin{align}\label{4.8}
|\T(z_{\e})|^{2^{*}_{s}-1}_{2^{*}_{s}-1}&=|\T(\varphi_{\e})-(\T(\varphi_{\e}))_{\Pi}|^{2^{*}_{s}-1}_{2^{*}_{s}-1} \nonumber\\
&\leq C_{6}(|\T(\varphi_{\e})|^{2^{*}_{s}-1}_{2^{*}_{s}-1}+|(\T(\varphi_{\e}))_{\Pi}|^{2^{*}_{s}-1}_{2^{*}_{s}-1}) \nonumber \\
&\leq C_{7}\e^{\frac{N-2s}{2}}
\end{align}
and
\begin{align}\label{4.9}
|\T(z_{\e})|_{1}\leq |\T(\varphi_{\e})|_{1}+|(\T(\varphi_{\e}))_{\Pi}|_{1}\leq C_{8} \e^{\frac{N-2s}{2}}. 
\end{align}
Now, by using (\ref{identity}) again, we derive that
\begin{align*}
&|\T(y+t z_{\e})|^{2^{*}_{s}}_{2^{*}_{s}, W}-|t\T(z_{\e})|_{2^{*}_{s}, W}^{2^{*}_{s}}-|\T(y)|_{2^{*}_{s}, W}^{2^{*}_{s}} \\
&=2^{*}_{s}\int_{0}^{1} d\tau \int_{\partial^{0}\mathcal{S}_{2\pi}} W(x)\left[|t\T(z_{\e})+\tau \T(y)|^{2^{*}_{s}-2}(t\T(z_{\e})+\tau \T(y))-|\tau \T(y)|^{2^{*}_{s}-2} \tau \T(y)\right] \T(y) dx \\
&=2^{*}_{s} (2^{*}_{s}-1)\int_{0}^{1} d\tau \int_{\partial^{0}\mathcal{S}_{2\pi}} W(x)|t\alpha(x)\T(z_{\e})+\tau \T(y)|^{2^{*}_{s}-2}(t\T(z_{\e})) \T(y) dx,
\end{align*}
for some measurable function $\alpha(x)$ such that $0<\alpha(x)<1$.
Using Young's inequality, estimates (\ref{4.8}) and (\ref{4.9}) and the fact that all norms in $\Y$ are equivalent, we deduce from the last inequality that
\begin{align*}
&\left||\T(y+t z_{\e})|^{2^{*}_{s}}_{2^{*}_{s}, W}-|t\T(z_{\e})|_{2^{*}_{s}, W}^{2^{*}_{s}}-|\T(y)|_{2^{*}_{s}, W}^{2^{*}_{s}}\right| \\
&\leq C_{9} \int_{0}^{1} d\tau \int_{\partial^{0}\mathcal{S}_{2\pi}} W(x)\left[|\T(y)| |t\T(z_{\e})|^{2^{*}_{s}-1}+\tau^{2^{*}_{s}-2} |t\T(z_{\e})| |\T(y)|^{2^{*}_{s}-1}\right] dx \\
&\leq C_{10}\left(|t\T(z_{\e})|_{2^{*}_{s}-1}^{2^{*}_{s}-1}|\T(y)|_{\infty, W}+ |t\T(z_{\e})|_{1} |\T(y)|^{2^{*}_{s}-1}_{\infty, W}\right) \\
&\leq C_{11}\left(|t\T(z_{\e})|_{2^{*}_{s}-1}^{2^{*}_{s}-1}|\T(y)|_{2^{*}_{s}, W}+ |t\T(z_{\e})|_{1}|\T(y)|^{2^{*}_{s}-1}_{2^{*}_{s}, W}\right) \\
&\leq C_{12} (C_{13} t^{{2^{*}_{s}}-1} \e^{\frac{N-2s}{2}} |\T(y)|_{2^{*}_{s}, W}+C_{14} t^{2^{*}_{s}} \e^{N})+\frac{1}{4}|\T(y)|_{2^{*}_{s}, W}^{2^{*}_{s}} \\ 
&\leq \frac{1}{2} |\T(y)|_{2^{*}_{s}, W}^{2^{*}_{s}}+C_{15} t^{2^{*}_{s}} \e^{\frac{N(N-2s)}{N+2s}}
\end{align*}
which completes the proof of Lemma.

Finally, we note that the above estimate gives the following inequality which we will use later
\begin{align}\label{4.11}
&\left| |\T(y+t z_{\e})|^{2^{*}_{s}}_{2^{*}_{s}, W}-|t\T(z_{\e})|_{2^{*}_{s}, W}^{2^{*}_{s}}-|\T(y)|_{2^{*}_{s}, W}^{2^{*}_{s}}\right| \nonumber\\
&\quad \quad \quad \quad \leq C_{16} \left(t^{{2^{*}_{s}}-1} \e^{\frac{N-2s}{2}} |\T(y)|_{2^{*}_{s}}+t^{2^{*}_{s}} \e^{N}\right)+\frac{1}{4}|\T(y)|_{2^{*}_{s}, W}^{2^{*}_{s}} {\color{red}{.}}
\end{align}
\end{proof}

\begin{lem}\label{lemimportant}
Let $z_{\e}$ be the function defined as in Lemma \ref{lemma2}.
Then we have 
\begin{align}\label{important1}
\|z_{\e}\|_{\X}^{2}- m^{2s} |\T(z_{\e})|_{L^{2}(0,T)^{N}}^{2}\leq \iint_{\mathcal{S}_{2\pi}}\xi^{1-2s} |\nabla \varphi_{\e}|^{2} dx d\xi.
\end{align}
\end{lem}
\begin{proof}
Firstly, we can see that
\begin{align*}
&\nabla_{x}z_{\e}=\nabla_{x}\varphi_{\e} \,\theta(m\xi)\\
&\partial_{\xi}z_{\e}=[\partial_{\xi}\varphi_{\e} - (\partial_{\xi}\varphi_{\e})_{\Pi}]\theta(m\xi)+[\varphi_{\e}-(\varphi_{\e})_{\Pi}]\theta'(m\xi)m. 
\end{align*}
Hence we get
\begin{align*}
&\|z_{\e}\|_{\X}^{2}- m^{2s} |\T(z_{\e})|_{2}^{2} = \\
&=\iint_{\mathcal{S}_{2\pi}} \xi^{1-2s} [\, |\nabla_{x}\varphi_{\e}|^{2} \theta^{2}(m\xi) + |\partial_{\xi} \varphi_{\e} - (\partial_{\xi} \varphi_{\e})_{\Pi}|^{2} \theta^{2}(m\xi) +m^{2} |\varphi_{\e} - (\varphi_{\e})_{\Pi}|^{2}(\theta'(m\xi))^{2} \\
&+ 2m \theta'(m\xi) \theta(m\xi) [\varphi_{\e} - (\varphi_{\e})_{\Pi}]\partial_{\xi}[\varphi_{\e} -(\varphi_{\e})_{\Pi}] + m^{2} |\varphi_{\e}- (\varphi_{\e})_{\Pi}|^{2} \theta^{2}(m\xi)\, ] \,dxd\xi \\
&- m^{2s} \int_{\partial^{0}\mathcal{S}_{2\pi}} |\T(\varphi_{\e}) - (\T(\varphi_{\e}))_{\Pi}|^{2} dx.
\end{align*}
Taking into account $\theta''+\frac{1-2s}{\xi}\theta'=\theta$, $\theta(0)=1$, $\theta(\infty)=0$ and $-\xi^{1-2s}\theta'(\xi)\rightarrow \kappa_{s}\equiv 1$ as $\xi\rightarrow 0^{+}$, we have
\begin{align*}
&\iint_{\mathcal{S}_{2\pi}} \xi^{1-2s} \Bigl[\, m^{2} |\varphi_{\e} - (\varphi_{\e})_{\Pi}|^{2}(\theta'(m\xi))^{2} + 2m \theta'(m\xi) \theta(m\xi) [\varphi_{\e} - (\varphi_{\e})_{\Pi}]\partial_{\xi}[\varphi_{\e} -(\varphi_{\e})_{\Pi}]\\
& + m^{2} |\varphi_{\e}- (\varphi_{\e})_{\Pi}|^{2} \theta^{2}(m\xi)\, \Bigr] \,dxd\xi - m^{2s} \int_{\partial^{0}\mathcal{S}_{2\pi}} |\T(\varphi_{\e}) - (\T(\varphi_{\e}))_{\Pi}|^{2} dx\\
&=\iint_{\mathcal{S}_{2\pi}} \xi^{1-2s} \Bigl[\, m^{2} |\varphi_{\e} - (\varphi_{\e})_{\Pi}|^{2}(\theta'(m\xi))^{2}+ m^{2} |\varphi_{\e}- (\varphi_{\e})_{\Pi}|^{2} \theta^{2}(m\xi)\, \Bigr] \,dxd\xi \\
&+\iint_{\mathcal{S}_{2\pi}} \xi^{1-2s} \partial_{\xi} (|\varphi_{\e} - (\varphi_{\e})_{\Pi}|^{2}) m \theta(m\xi) \theta'(m\xi) \, dxd\xi - m^{2s} \int_{\partial^{0}\mathcal{S}_{2\pi}} |\T(\varphi_{\e}) - (\T(\varphi_{\e}))_{\Pi}|^{2} dx\\
&=\iint_{\mathcal{S}_{2\pi}} \xi^{1-2s} \Bigl[\, m^{2} |\varphi_{\e} - (\varphi_{\e})_{\Pi}|^{2}(\theta'(m\xi))^{2}+ m^{2} |\varphi_{\e}- (\varphi_{\e})_{\Pi}|^{2} \theta^{2}(m\xi)\, \Bigr] \,dxd\xi \\
&+ m^{2s} \int_{\partial^{0}\mathcal{S}_{2\pi}} |\T(\varphi_{\e}) - (\T(\varphi_{\e}))_{\Pi}|^{2} dx - \iint_{\mathcal{S}_{2\pi}} |\varphi_{\e} - (\varphi_{\e})_{\Pi}|^{2} \Bigl[(1-2s)\xi^{-2s}m \theta(m\xi)\theta'(m\xi)  \\
&+ \xi^{1-2s} m^{2}(\theta'(m\xi))^{2}+\xi^{1-2s}m^{2}\theta(m\xi) \theta''(m\xi)\Bigr]\, dxd\xi -m^{2s} \int_{\partial^{0}\mathcal{S}_{2\pi}} |\T(\varphi_{\e}) - (\T(\varphi_{\e}))_{\Pi}|^{2} dx\\
&=\iint_{\mathcal{S}_{2\pi}} \xi^{1-2s} [\, m^{2} |\varphi_{\e} - (\varphi_{\e})_{\Pi}|^{2}(\theta'(m\xi))^{2}+ m^{2} |\varphi_{\e}- (\varphi_{\e})_{\Pi}|^{2} \theta^{2}(m\xi)\, ] \,dxd\xi \\
&-\iint_{\mathcal{S}_{2\pi}} |\varphi_{\e} - (\varphi_{\e})_{\Pi}|^{2} [\xi^{1-2s} m^{2}(\theta'(m\xi))^{2}+\xi^{1-2s}m^{2}\theta^{2}(m\xi)]\, dxd\xi =0.
\end{align*}
As a consequence
\begin{align}\label{important2}
\|z_{\e}\|_{\X}^{2}- m^{2s} |\T(z_{\e})|_{2}^{2}=\iint_{\mathcal{S}_{2\pi}} \xi^{1-2s} \theta^{2}(m\xi) [\, |\nabla_{x}\varphi_{\e}|^{2} + |\partial_{\xi} \varphi_{\e} - (\partial_{\xi} \varphi_{\e})_{\Pi}|^{2} ]\, dxd\xi.
\end{align}
Since
\begin{align*}
\int_{\partial^{0}\mathcal{S}_{2\pi}}|\partial_{\xi} \varphi_{\e} - (\partial_{\xi} \varphi_{\e})_{\Pi}|^{2} \, dx &=\int_{\partial^{0}\mathcal{S}_{2\pi}} |\partial_{\xi} \varphi_{\e}|^{2}\, dx+\int_{\partial^{0}\mathcal{S}_{2\pi}}| (\partial_{\xi} \varphi_{\e})_{\Pi}|^{2}\, dx-2 (\partial_{\xi} \varphi_{\e})_{\Pi} \int_{\partial^{0}\mathcal{S}_{2\pi}} \partial_{\xi} \varphi_{\e}\, dx \\
&= \int_{\partial^{0}\mathcal{S}_{2\pi}} [|\partial_{\xi} \varphi_{\e}|^{2} - |(\partial_{\xi} \varphi_{\e})_{\Pi}|^{2} ]\, dx,
\end{align*}
we can deduce that 
\begin{align}\label{important3}
&\iint_{\mathcal{S}_{2\pi}} \xi^{1-2s} \theta^{2}(m\xi) [\, |\nabla_{x}\varphi_{\e}|^{2} + |\partial_{\xi} \varphi_{\e} - (\partial_{\xi} \varphi_{\e})_{\Pi}|^{2} ]\, dxd\xi \nonumber \\
&= \iint_{\mathcal{S}_{2\pi}} \xi^{1-2s} \theta^{2}(m\xi) [\, |\nabla_{x}\varphi_{\e}|^{2} +  [|\partial_{\xi} \varphi_{\e}|^{2} - |(\partial_{\xi} \varphi_{\e})_{\Pi}|^{2} ]\, dxd\xi \nonumber \\ 
&\leq \iint_{\mathcal{S}_{2\pi}} \xi^{1-2s} \theta^{2}(m\xi) [\, |\nabla_{x}\varphi_{\e}|^{2} + |\partial_{\xi} \varphi_{\e}|^{2} ]\, dxd\xi \nonumber \\ 
&\leq \iint_{\mathcal{S}_{2\pi}} \xi^{1-2s} [\, |\nabla_{x}\varphi_{\e}|^{2} + |\partial_{\xi} \varphi_{\e}|^{2} ]\, dxd\xi=\iint_{\mathcal{S}_{2\pi}} \xi^{1-2s}  |\nabla \varphi_{\e}|^{2} \, dxd\xi,
\end{align}
where in the last inequality we have used the fact that $0<\theta(t)\leq 1$ \cite{Erd}.
Putting together (\ref{important2}) and (\ref{important3}) we deduce that (\ref{important1}) holds.

\end{proof}

\begin{lem}\label{Linking2}
Let $Q_{\e}:=\{y+t z_{\e}: y\in \Y, t\geq 0\}$ be the set defined as in Lemma \ref{lemma2}.
Then we have 
$$
\max_{v \in Q_{\e}} \J_{m}(v)<\frac{s}{N} W(0)^{-\frac{N-2s}{2s}}S_{*}^{\frac{N}{2s}}.
$$
\end{lem}
\begin{proof}
Let us observe that for every $v\in \X-\{0\}$
\begin{align*}
\J_{m}(tv)&=\frac{t^{2}}{2} \left[\|v\|^{2}_{\X}-m^{2s}|\T(v)|_{2}^{2}\right]-\frac{t^{2^{*}_{s}}}{2^{*}_{s}}|\T(v)|^{2^{*}_{s}}_{2^{*}_{s}, W}-\int_{\partial^{0}\mathcal{S}_{2\pi}} F(x, t\T(v)) dx \\
&\leq \frac{t^{2}}{2} \left[\|v\|^{2}_{\X}-m^{2s}|\T(v)|_{2}^{2}\right]-\frac{t^{2^{*}_{s}}}{2^{*}_{s}}|\T(v)|^{2^{*}_{s}}_{2^{*}_{s}, W}\rightarrow -\infty \mbox{ as } t\rightarrow \infty.
\end{align*}
Then, we can find $t_{\e}\geq 0$ such that 
$$
\J_{m}(t_{\e} v)=\sup_{t\geq 0} \J_{m}(tv).
$$
We can assume that $t_{\e}>0$, and since it satisfies 
$$
t_{\e}\left[\|v\|^{2}_{\X}-m^{2s}|\T(v)|_{2}^{2}\right]-t_{\e}^{2^{*}_{s}-1}|\T(v)|^{2^{*}_{s}}_{2^{*}_{s}, W}-\int_{\partial^{0}\mathcal{S}_{2\pi}} \T(v)f(x, t_{\e}\T(v)) dx=0,
$$
we obtain
$$
t_{\e}\leq \left[\frac{\|v\|_{\X}^{2}- m^{2s} |\T(v)|_{2}^{2}}{|\T(v)|^{2^{*}_{s}}_{2^{*}_{s}, W}}\right]^{\frac{N-2s}{4s}}=:A.
$$
Let us note that the function
$$
t\mapsto \frac{t^{2}}{2}\left[\|v\|^{2}_{\X}-m^{2s}|\T(v)|_{2}^{2}\right]-\frac{t^{2^{*}_{s}}}{2^{*}_{s}}|\T(v)|_{2^{*}_{s}, W}^{2^{*}_{s}}
$$
is increasing in $[0, A]$, so we can deduce that
\begin{align}\label{10}
\mathcal{J}_{m}(t_{\e}v)\leq \frac{s}{N}\left[\frac{\|v\|_{\X}^{2}- m^{2s} |\T(v)|_{2}^{2}}{|\T(v)|^{2}_{2^{*}_{s}, W}}\right]^{\frac{N}{2s}}- \int_{\partial^{0}\mathcal{S}_{2\pi}} F(x, t_{\e}\T(v)) dx.
\end{align}
Now, fix $v=y+t z_{\e}\in Q_{\e}$ such that $|\T(v)|_{2^{*}_{s}, W}=1$. Hence, by using Theorem \ref{thm6}, we can see that 
\begin{align}\label{monster}
\|v\|_{\X}^{2}- m^{2s} |\T(v)|_{2}^{2}&=\| y\|_{\X}^{2}- m^{2s} |\T(y)|^{2}_{2}+t^{2}[\| z_{\e}\|_{\X}^{2}- m^{2s} |\T(z_{\e})|^{2}_{2}]\nonumber \\
&=t^{2}[\| z_{\e}\|_{\X}^{2}- m^{2s} |\T(z_{\e})|^{2}_{2}] \nonumber \\
&=\frac{\| z_{\e}\|_{\X}^{2}- m^{2s} |\T(z_{\e})|^{2}_{2}}{| \T(z_{\e})|_{2^{*}_{s}, W}^{2}}\, |t \T(z_{\e})|_{2^{*}_{s}, W}^{2}. 
\end{align}
By using (\ref{CY2star}), we get
\begin{equation}\label{mediafi1}
\left|\int_{\partial^{0}\mathcal{S}_{2\pi}} W(x)(|\T(z_{\e})|^{2^{*}_{s}} - |\T(\varphi_{\e})|^{2^{*}_{s}})\, dx \right|\leq C \e^{N-2s},
\end{equation}
which together with $N-2s\geq  2s$ and Lemma \ref{lemma1}-(ii), yields
\begin{align}\label{mediafi2}
|\T(z_{\e})|^{2}_{2^{*}_{s}, W}&=(|\T(z_{\e})|^{2^{*}_{s}}_{2^{*}_{s}, W})^{\frac{2}{2^{*}_{s}}}=(|\T(\varphi_{\e})|^{2^{*}_{s}}_{2^{*}_{s}, W}+O(\e^{N-2s}))^{\frac{N-2s}{N}} \nonumber\\
&=(W(0)S_{*}^{\frac{N}{2s}}+O(\e^{N})+O(\e^{2s})+O(\e^{N-2s}))^{\frac{N-2s}{N}} \nonumber\\
&=W(0)^{\frac{N-2s}{N}}S_{*}^{\frac{N-2s}{2s}}+O(\e^{2s \frac{(N-2s)}{N}}).
\end{align}
On the other hand, in view of Lemma \ref{lemimportant}, we know that
\begin{align}\label{important11}
\|z_{\e}\|_{\X}^{2}- m^{2s} |\T(z_{\e})|_{2}^{2}\leq \iint_{\mathcal{S}_{2\pi}}\xi^{1-2s} |\nabla \varphi_{\e}|^{2} dx d\xi.
\end{align}
Thus, putting together (\ref{monster}), Lemma \ref{lemma1}-(i), (\ref{mediafi2}) and (\ref{important11}), we can see that
\begin{align}\label{14}
\|v\|_{\X}^{2}- m^{2s} |\T(v)|_{2}^{2} \leq [W(0)^{-\frac{N-2s}{N}}S_{*}+O(\e^{2s\frac{(N-2s)}{N}})] |t \T(z_{\e})|_{2^{*}_{s}, W}^{2}. 
\end{align}
Now, by using Lemma \ref{lemma2}, $|\T(v)|_{2^{*}_{s}, W}=1$ and (\ref{mediafi1}), we have
\begin{align}\label{15}
1=|\T(v)|_{2^{*}_{s}, W}^{2^{*}_{s}}  &\geq |t \T(z_{\e})|^{2^{*}_{s}}_{2^{*}_{s}, W}+\frac{1}{2} |\T(y)|^{2^{*}_{s}}_{2^{*}_{s}, W} -K_{5} t^{2^{*}_{s}} \e^{\frac{N(N-2s)}{N+2s}} \nonumber \\
&\geq t^{2^{*}_{s}} |\varphi_{\e}(\cdot, 0)|^{2^{*}_{s}}_{2^{*}_{s}, W}-C_{5} t^{2^{*}_{s}}\e^{N-2s}+\frac{1}{2} |\T(y)|^{2^{*}_{s}}_{2^{*}_{s}, W}-K_{5} t^{2^{*}_{s}} \e^{\frac{N(N-2s)}{N+2s}},
\end{align}
which implies that $t$ is bounded.

We distinguish two cases.
Firstly, we suppose that $|\T(y)|_{2^{*}_{s}, W}^{2^{*}_{s}}\leq 2K_{5} t^{2^{*}_{s}} \e^{N\frac{(N-2s)}{N+2s}}$.  Then by (\ref{4.11}) in Lemma \ref{lemma2}, we have the following estimate
\begin{align*}
|t\T(z_{\e})|_{2^{*}_{s}, W}^{2^{*}_{s}}\leq 1-\frac{3}{4}|\T(y)|_{2^{*}_{s}, W}^{2^{*}_{s}}+K_{6}(\e^{\frac{N-2s}{2}}|\T(y)|_{2}+\e^{N})
\end{align*}
which yields
\begin{align}\label{16}
|t\T(z_{\e})|_{2^{*}_{s}, W}^{2}&\leq \left[1-\frac{3}{4}|\T(y)|_{2^{*}_{s}, W}^{2^{*}_{s}}+K_{7}(\e^{\frac{N-2s}{2}}|\T(y)|_{2}+\e^{N})\right]^{\frac{2}{2^{*}_{s}}} \nonumber\\
&\leq 1+K_{8} (\e^{\frac{N-2s}{2}}|\T(y)|_{2}+\e^{N}) \nonumber\\
&\leq 1+K_{9}\e^{N\frac{(N-2s)}{N+2s}},
\end{align}
where in the last inequality we used H\"older inequality, Young's inequality with exponents $\frac{2^{*}_{s}}{2^{*}_{s}-1}$ and $2^{*}_{s}$, $|\T(y)|_{2^{*}_{s}, W}^{2^{*}_{s}}\leq 2K_{5} t^{2^{*}_{s}} \e^{N\frac{(N-2s)}{N+2s}}$ and the boundedness of $t$ to infer that 
$$
\e^{\frac{N-2s}{2}}|\T(y)|_{2}\leq C\e^{\frac{N-2s}{2}}|\T(y)|_{2^{*}_{s}, W}\leq C'\e^{\frac{N(N-2s)}{N+2s}}+C''\e^{\frac{N(N-2s)}{N+2s}}=C'''\e^{\frac{N(N-2s)}{N+2s}}.
$$.\\
If we assume that $|\T(y)|_{2^{*}_{s}, W}^{2^{*}_{s}}> 2K_{5} t^{2^{*}_{s}} \e^{N\frac{(N-2s)}{N+2s}}$, from the inequality (\ref{15}), we deduce easily that
\begin{align}\label{17}
|t\T(z_{\e})|_{2^{*}_{s}, W}^{2^{*}_{s}}\leq 1.
\end{align}
As a consequence of (\ref{16}) and (\ref{17}), we get 
\begin{align}\label{17bis}
|t\T(z_{\e})|_{2^{*}_{s}, W}^{2^{*}_{s}}\leq 1+K_{9}\e^{N\frac{(N-2s)}{N+2s}}.
\end{align}
Since $t_{\e}$ satisfies 
\begin{align*}
\left[\|y+t z_{\e}\|^{2}_{\X}-m^{2s}|\T(y+t z_{\e})|_{2}^{2}\right]&-t_{\e}^{2^{*}_{s}-2}|\T(y+tz_{\e})|^{2^{*}_{s}}_{2^{*}_{s}, W} \\
&-\int_{\partial^{0}\mathcal{S}_{2\pi}} \frac{\T(y+t z_{\e})f(x, \T(y+t_{\e}z_{\e}))}{t_{\e}} dx=0,
\end{align*}
we obtain
\begin{equation}\label{17Bbis}
\lim_{\e\rightarrow 0} \left[\|y+t z_{\e}\|^{2}_{\X}-m^{2s}|\T(y+t z_{\e})|_{2}^{2}\right]\geq \lim_{\e\rightarrow 0} t_{\e}^{2^{*}_{s}-2}.
\end{equation}
Taking into account (\ref{14}), (\ref{17bis}) and (\ref{17Bbis}), we can infer that
$$
\lim_{\e\rightarrow 0} t_{\e}^{2^{*}_{s}-2} \leq W(0)^{-\frac{N-2s}{N}}S_{*},
$$
that is $t_{\e}$ is bounded for any $\e>0$ small enough. Hence, we may assume that $t_{\e}\rightarrow t_{0}\geq 0$ as $\e\rightarrow 0$. If $t_{0}=0$, we have finished. Thus, we suppose that $t_{0}>0$.

Now, we estimate the integral involving $F$. 
Then, we have
\begin{align}\begin{split}\label{18}
&\left| \int_{\partial^{0}\mathcal{S}_{2\pi}} F(x, \T(y+t z_{\e})) \, dx - \int_{\partial^{0}\mathcal{S}_{2\pi}} F(x,\T(y))\, dx -\int_{\partial^{0}\mathcal{S}_{2\pi}} F(x, t \T(z_{\e}))\, dx \right|\\
&= \left| \int_{\partial^{0}\mathcal{S}_{2\pi}} \left[ \int_{0}^{t\T(z_{\e})} f(x, \T(y)+\tau)\, d\tau - \int_{0}^{t\T(z_{\e})} f(x, \tau)\, d\tau\right]\, dx\right|\\
&\leq K_{10}\left[ \int_{\partial^{0}\mathcal{S}_{2\pi}} |(t\T(z_{\e}))|(1+|\T(y)+t\T(z_{\e})|^{p-1}) \, dx + \int_{\partial^{0}\mathcal{S}_{2\pi}} |(t\T(z_{\e}))| (1+|t\T(z_{\e})|^{p-1}) \, dx\right]\\
&\leq K_{10} \left[ \int_{\partial^{0}\mathcal{S}_{2\pi}} (|\T(y)|^{p-1} |t\T(z_{\e})| + |t\T(z_{\e})| + |t\T(z_{\e})|^{p})\, dx \right]. 
\end{split}\end{align}
It is clear that the condition $|\T(y+t z_{\e})|_{2^{*}_{s}}=1$ implies that $|\T(y)|_{\infty}$ is uniformly bounded.  
Arguing as in Lemma \ref{lemma2} (see formula \eqref{CY2star} there), we can see that it holds 
\begin{align*}
\left| \int_{\partial^{0}\mathcal{S}_{2\pi}} (|\T(z_{\e})|^{p}- |\T(\varphi_{\e})|^{p})\, dx \right| &\leq K_{11} (|\T(\varphi_{\e})|_{p-1}^{p-1} |(\T(\varphi_{\e}))_{\Pi}|_{\infty} + |(\T(\varphi_{\e}))_{\Pi}|_{p}^{p})\\
&\leq K_{12}(\e^{N- \frac{(N-2s)(p-1)}{2}} \e^{\frac{N-2s}{2}} + \e^{\frac{p(N-2s)}{2}})\\
&\leq O(\e^{\frac{N-2s}{2}}). 
\end{align*}
This and (\ref{18}) yield
\begin{equation}\label{19}
\left|\int_{\partial^{0}\mathcal{S}_{2\pi}} [F(x, \T(v)) - F(x,\T(y))- F(x, t\T(z_{\e}))]\, dx \right| \leq K_{13} (\e^{\frac{N-2s}{2}} + \e^{N- \frac{p(N-2s)}{2}}).
\end{equation}
Putting together (\ref{10}), (\ref{14}), (\ref{17bis}) and (\ref{19}), it follows that
\begin{align}\begin{split}\label{20}
\mathcal{J}_{m}(t_{\e} (y+t z_{\e}))
&\leq \frac{s}{N} W(0)^{-\frac{N-2s}{2s}} S_{*}^{\frac{N}{2s}} +  O(\e^{N-2s})+O(\e^{\frac{N-2s}{2}}) + O(\e^{N-\frac{p(N-2s)}{2}}) -\int_{\partial^{0}\mathcal{S}_{2\pi}} F(x, t_{\e}\T(y))\, dx \\
&- \int_{\partial^{0}\mathcal{S}_{2\pi}} F(x, t_{\e} t \T(z_{\e}))\, dx\\
&\leq \frac{s}{N} W(0)^{-\frac{N-2s}{2s}} S_{*}^{\frac{N}{2s}} +  O(\e^{\frac{N-2s}{2}}) + O(\e^{N-\frac{p(N-2s)}{2}}) - \int_{\partial^{0}\mathcal{S}_{2\pi}} F(x, t_{\e} t \T(z_{\e}))\, dx. 
\end{split}\end{align}
Now, we observe that
\begin{align}\begin{split}\label{21}
\left| \int_{\partial^{0}\mathcal{S}_{2\pi}} F(x, t_{\e}t \T(z_{\e})) - F(x, t_{\e} t \T(\varphi_{\e}))\, dx \right|&\leq \int_{\partial^{0}\mathcal{S}_{2\pi}} \left|\int_{t_{\e} t\T(\varphi_{\e})}^{t_{\e}t\T(z_{\e})} f(x, \tau)\, d\tau \right| \, dx \\
&\leq K_{14}(|\T(z_{\e})|_{2}^{2} + |\T(z_{\e})|_{p}^{p})= O(\e^{\frac{N-2s}{2}}). 
\end{split}\end{align}
Hence, by using (\ref{20}), (\ref{21}), $(f6)$ and $t_{\e}\rightarrow t_{0}>0$, we get
\begin{align*}
\mathcal{J}_{m}(t_{\e} (y+t z_{\e}))&\leq \frac{s}{N} W(0)^{-\frac{N-2s}{2s}} S_{*}^{\frac{N}{2s}} + O(\e^{\frac{N-2s}{2}}) + O(\e^{N-\frac{p(N-2s)}{2}}) - \int_{\partial^{0}\mathcal{S}_{2\pi}} \overline{F}(t_{\e}t\T(z_{\e}))\, dx\\
&\leq \frac{s}{N} W(0)^{-\frac{N-2s}{2s}} S_{*}^{\frac{N}{2s}} + O(\e^{\frac{N-2s}{2}}) + O(\e^{N-\frac{p(N-2s)}{2}}) - \int_{B(0,R)} \overline{F}\left(\frac{C\e^{\frac{N-2s}{2}}}{(\e^2 +|x|^{2})^{\frac{N-2s}{2}}}\right)\, dx, 
\end{align*}
for some $C>0$ and $R>0$.
Since the assumption $(f6)$ implies that
\begin{align*}
\lim_{\e\rightarrow 0} \frac{1}{\e^{\frac{N-2s}{2}}} \int_{B(0,R)}\overline{F}\left( \frac{C\e^{\frac{N-2s}{2}}}{(\e^2 +|x|^{2})^{\frac{N-2s}{2}}}\right)\, dx=\infty
\end{align*}
and 
\begin{align*}
\lim_{\e\rightarrow 0} \frac{1}{\e^{N-\frac{p(N-2s)}{2}}} \int_{B(0,R)}\overline{F}\left(\frac{C\e^{\frac{N-2s}{2}}}{(\e^2 +|x|^{2})^{\frac{N-2s}{2}}}\right)\, dx=\infty,
\end{align*}
we can infer that
\begin{align*}
\mathcal{J}_{m}(t_{\e}(y+t z_{\e}))< \frac{s}{N}W(0)^{-\frac{N-2s}{2s}} S_{*}^{\frac{N}{2s}} . 
\end{align*}
This ends the proof of lemma.

\end{proof}

\noindent
To obtain the existence of a critical value of $\mathcal{J}_{m}$, we need to prove the Palais-Smale condition.
This condition will be satisfied for all $c \in \R$ such that $c<\frac{s}{N} W(0)^{-\frac{N-2s}{2s}} S_{*}^{\frac{N}{2s}}$,
where
$$
S_{*}=\inf \left\{ \frac{\iint_{\R^{N+1}_{+}} \xi^{1-2s} |\nabla v|^{2} dx d\xi}{\Bigl(\int_{\R^{N}} |v(x, 0)|^{2^{*}_{s}} \,dx\Bigr)^{\frac{2}{2^{*}_{s}}}}: v \in \mathcal{C}^{\infty}_{c}(\overline{\R^{N+1}_{+}}) \right \} 
$$
is the best constant of the fractional Sobolev embedding $H^{s}(\R^{N})$ into $L^{2^{*}_{s}}(\R^{N})$; see \cite{CT}.
\begin{lem}\label{psthm}
Let $c\in \R$ be such that $c<c^{*}:=\frac{s}{N} W(0)^{-\frac{N-2s}{2s}} S_{*}^{\frac{N}{2s}}$ and let $(v_{j})\subset \X$ be a sequence such that
\begin{equation}\label{pscs}
\mathcal{J}_{m}(v_{j})\rightarrow c \mbox{ and } \mathcal{J}_{m}'(v_{j})\rightarrow 0 \mbox{ as } j\rightarrow \infty. 
\end{equation}
Then $(v_{j})$  has a  strongly convergent subsequence in $\X$.
\end{lem}
\begin{proof}
By using $(f5)$, we have for $j$ large
\begin{align}\begin{split}\label{B}
&\frac{s}{N} W(0)^{-\frac{N-2s}{2s}} S_{*}^{\frac{N}{2s}}+1+\|v_{j}\|_{\X} \\
&\geq \mathcal{J}_{m}(v_{j})-\frac{1}{2} \langle \mathcal{J}_{m}'(v_{j}), v_{j}\rangle \\
&=\Bigl(\frac{1}{2}-\frac{1}{2^{*}_{s}}\Bigr) |\T(v_{j})|^{2^{*}_{s}}_{2^{*}_{s}, W} +\frac{1}{2}\int_{\partial^{0}\mathcal{S}_{2\pi}}  f(x, \T(v_{j})) \T(v_{j}) dx -\int_{\partial^{0}\mathcal{S}_{2\pi}}  F(x, \T(v_{j})) dx \\
& \geq \Bigl(\frac{1}{2}-\frac{1}{2^{*}_{s}}\Bigr) |\T(v_{j})|^{2^{*}_{s}}_{2^{*}_{s}, W}+\left(\frac{1}{2}-\frac{1}{\mu}\right)\int_{\partial^{0}\mathcal{S}_{2\pi}}  f(x, \T(v_{j})) \T(v_{j}) dx  \\
&\geq \Bigl(\frac{1}{2}-\frac{1}{2^{*}_{s}}\Bigr) |\T(v_{j})|^{2^{*}_{s}}_{2^{*}_{s},W}.
\end{split}\end{align}
Now, let us recall that $W(x)\geq \min_{x\in [-\pi, \pi]^{N}} W(x)>0$ in view of $(W1)$.
%

Then, by using H\"older inequality, we have the following estimate
\begin{align*}
\int_{\partial^{0}\mathcal{S}_{2\pi}} |\T(v_{j})|^{2} dx
\leq  |\partial^{0}\mathcal{S}_{2\pi}|^{\frac{2^{*}_{s}-2}{2^{*}_{s}}}  |\T(v_{j})|^{2}_{2^{*}_{s},W} \left(\min_{x\in[-\pi, \pi]^{N}} W(x)\right)^{-\frac{2}{2^{*}_{s}}},
\end{align*}
which together with (\ref{B}) yields
\begin{align*}
|\T(v_{j})|^{2^{*}_{s}}_{2}\leq c_{1} |\T(v_{j})|^{2^{*}_{s}}_{2^{*}_{s},W}\leq c_{2}(1+\|v_{j}\|_{\X}).
\end{align*}
On the other hand, by applying $(f5)$ and \eqref{B}, we can see that 
\begin{align*}
\int_{\partial^{0}\mathcal{S}_{2\pi}}  F(x, \T(v_{j})) dx\leq \mu^{-1} \int_{\partial^{0}\mathcal{S}_{2\pi}}  f(x, \T(v_{j})) \T(v_{j}) dx \leq c_{3}(1+\|v_{j}\|_{\X}).
\end{align*}
Summing up, we have 
\begin{align*}
\|v_{j}\|_{\X}^{2} &=2\mathcal{J}_{m}(v_{j})+m^{2s} |\T(v_{j})|^{2}_{2}+\frac{2}{2^{*}_{s}} |\T(v_{j})|^{2^{*}_{s}}_{2^{*}_{s}, W}+2 \int_{\partial^{0}\mathcal{S}_{2\pi}}  F(x, \T(v_{j})) dx \\
& \leq c_{4} +c_{5} (1+\|v_{j}\|_{\X})^{\frac{2}{2^{*}_{s}}}+c_{6}(1+\|v_{j}\|_{\X}),
\end{align*}
that is $(v_{j})_{j}$ is bounded in ${\mathbb{X}^{s}_{m}}$. 

By Theorem \ref{thm2}, we can extract a subsequence, which we denote again by $v_{j}$, such that
\begin{align}\begin{split}\label{conv}
v_{j} & \rightharpoonup v  \mbox{ in } \X \\
\T(v_{j}) & \rightarrow \T(v)   \mbox{ in } L^{q}(-\pi, \pi)^{N} \mbox{ for } q\in [1, 2^{*}_{s})\\
\T(v_{j})  & \rightarrow \T(v)   \mbox{ a.e. in }  (-\pi, \pi)^{N}.  
\end{split}\end{align}
Moreover, by using (\ref{conv}) and $(f2)$-$(f4)$, we can see that
\begin{equation}\label{pf1}
\int_{\partial^{0}\mathcal{S}_{2\pi}}  F(x, \T(v_{j})) dx\rightarrow \int_{\partial^{0}\mathcal{S}_{2\pi}}  F(x, \T(v)) dx,
\end{equation}
\begin{equation}\label{pf2}
\int_{\partial^{0}\mathcal{S}_{2\pi}}  f(x, \T(v_{j})) (\T(v_{j})-\T(v)) dx \rightarrow 0
\end{equation}
and
\begin{equation}\label{pf3}
\int_{\partial^{0}\mathcal{S}_{2\pi}}  f(x, \T(v)) (\T(v_{j})-\T(v)) dx \rightarrow 0
\end{equation}
as $j\rightarrow \infty$.

\noindent
Hence, for every $\phi \in \X$, we obtain that, as $j \rightarrow \infty$
\begin{align}
&\iint_{\mathcal{S}_{2\pi}} \xi^{1-2s} (\nabla v_{j} \nabla \phi+m^{2}v_{j} \phi) \,dxd\xi-m^{2s} \int_{\partial^{0}\mathcal{S}_{2\pi}} \T(v_{j}) \T(\phi) \,dx  \nonumber \\
& -\int_{\partial^{0}\mathcal{S}_{2\pi}} W(x) |\T(v_{j})|^{2^{*}_{s}-2}\T(v_{j}) \T(\phi) \,dx -\int_{\partial^{0}\mathcal{S}_{2\pi}}  f(x, \T(v_{j})) \T(\phi) dx \nonumber\\
&\rightarrow \iint_{\mathcal{S}_{2\pi}} \xi^{1-2s}(\nabla v \nabla \phi+m^{2}v \phi) \,dxd\xi-m^{2s} \int_{\partial^{0}\mathcal{S}_{2\pi}} \T(v) \T(\phi) \,dx  \nonumber \\
&-\int_{\partial^{0}\mathcal{S}_{2\pi}} W(x) |\T(v)|^{2^{*}_{s}-2}\T(v) \T(\phi) \,dx-\int_{\partial^{0}\mathcal{S}_{2\pi}}  f(x, \T(v)) \T(\phi) dx.
\end{align}
Since $\mathcal{J}_{m}'(v_{j}) \rightarrow 0$, we deduce that $\langle \mathcal{J}_{m}'(v), \phi \rangle=0$, for every $\phi \in \X$.
Choosing $\phi=v$, we have
\begin{align}
0=\|v\|^{2}_{\X}-m^{2s} |\T(v)|^{2}_{2} - |\T(v)|^{2^{*}_{s}}_{2^{*}_{s}, W} -\int_{\partial^{0}\mathcal{S}_{2\pi}}  f(x, \T(v)) \T(v) dx
\end{align}
and, by using $(f5)$, we get 
\begin{equation}
\mathcal{J}_{m}(v)\geq \Bigl(\frac{1}{2}-\frac{1}{2^{*}_{s}}\Bigr) |\T(v)|^{2^{*}_{s}}_{2^{*}_{s}, W} +\left(\frac{1}{2}-\frac{1}{\mu}\right)\int_{\partial^{0}\mathcal{S}_{2\pi}}  f(x, \T(v)) \T(v) dx\geq 0.
\end{equation}
By using Brezis-Lieb Lemma \cite{BL}, we can see that
\begin{align}\label{a2}
\|v_{j}\|_{\X}^{2} = \|v_{j}-v\|_{\X}^{2}+ \|v\|_{\X}^{2}+o(1)
\end{align}
and
\begin{align}\label{a1}
|\T(v_{j})|^{2^{*}_{s}}_{2^{*}_{s}, W} =|\T(v_{j})-\T(v)|^{2^{*}_{s}}_{2^{*}_{s}, W} + |\T(v)|^{2^{*}_{s}}_{2^{*}_{s}, W}+o(1).
\end{align}
Thus, by using (\ref{pf1}), (\ref{a2}), (\ref{a1}) and the fact that $\T(v_{j}) \rightarrow \T(v)$ in $L^{2}(-\pi,\pi)^{N}$, we have 
\begin{align}\label{C}
\mathcal{J}_{m}(v_{j})=\mathcal{J}_{m}(v)+\left[\frac{1}{2}\|v_{j}-v\|_{\X}^{2}-\frac{1}{2^{*}_{s}} |\T(v_{j})-\T(v)|^{2^{*}_{s}}_{2^{*}_{s}, W}\right]+o(1). 
\end{align}
Since $\T(v_{j}) \rightharpoonup \T(v)$ in $L^{2^{*}_{s}}(-\pi, \pi)^{N}$, we can see
\begin{align*}
\int_{\partial^{0}\mathcal{S}_{2\pi}} &W(x) (|\T(v_{j})|^{2^{*}_{s}-2} \T(v_{j})-|\T(v)|^{2^{*}_{s}-2} \T(v))(\T(v_{j})-\T(v)) \,dx  \nonumber \\
&=\int_{\partial^{0}\mathcal{S}_{2\pi}} W(x) (|\T(v_{j})|^{2^{*}_{s}}-|\T(v_{j})|^{2^{*}_{s}-2} \T(v_{j}) \T(v)) \,dx+o(1)  \nonumber \\
&=\int_{\partial^{0}\mathcal{S}_{2\pi}} W(x) (|\T(v_{j})|^{2^{*}_{s}}-|\T(v)|^{2^{*}_{s}})  \,dx+o(1) \nonumber \\
&=\int_{\partial^{0}\mathcal{S}_{2\pi}}W(x) |\T(v_{j})-\T(v)|^{2^{*}_{s}} \,dx+o(1),
\end{align*}
where we used \eqref{a1} in the last equality.
This, (\ref{pf2}) and (\ref{pf3}) yields
\begin{align*}
0&=\langle \mathcal{J}_{m}'(v_{j}), v_{j}-v \rangle \\
&=\langle \mathcal{J}_{m}'(v_{j})-\mathcal{J}_{m}'(v), v_{j}-v \rangle \\
&= \iint_{\mathcal{S}_{2\pi}} \xi^{1-2s}[|\nabla v_{j}-\nabla v|^{2}+m^{2}(v_{j}-v)^{2}] \,dxd\xi-\int_{\partial^{0}\mathcal{S}_{2\pi}} W(x) |\T(v_{j})-\T(v)|^{2^{*}_{s}} \,dx+o(1),
\end{align*}
that is 
\begin{equation}\label{A}
\|v_{j}-v\|^{2}_{\X}= |\T(v_{j})-\T(v)|^{2^{*}_{s}}_{2^{*}_{s}, W}+o(1).
\end{equation}
Taking into account $\mathcal{J}_{m}(v) \geq 0$  and (\ref{C}) we infer that 
\begin{align}
\frac{1}{2}\|v_{j}-v\|_{\X}^{2}-\frac{1}{2^{*}_{s}} |\T(v_{j})-\T(v)|^{2^{*}_{s}}_{2^{*}_{s}, W}&=\mathcal{J}_{m}(v_{j})-\mathcal{J}_{m}(v)+o(1) \nonumber \\
&\leq \mathcal{J}_{m}(v_{j})+o(1),
\end{align}
and by using (\ref{A}) and $c<\frac{s}{N} W(0)^{-\frac{N-2s}{2s}} S_{*}^{N/2s}$, we find
\begin{align}
\Bigl(\frac{1}{2}-\frac{1}{2^{*}_{s}}\Bigr)\|v_{j}-v\|_{\X}^{2} +o(1)\leq \mathcal{J}_{m}(v_{j})+o(1)=c<\frac{s}{N} W(0)^{-\frac{N-2s}{2s}} S_{*}^{N/2s}.
\end{align}
Since $\frac{1}{2}-\frac{1}{2^{*}_{s}}=\frac{s}{N}$, we get for all $j \geq j_{0}$
\begin{equation}\label{D1}
\|v_{j}-v\|_{\X}^{2}<W(0)^{-\frac{N-2s}{2s}} S_{*}^{N/2s}.
\end{equation}
Now, recalling (see \cite{HZ}) that for any $\varepsilon>0$ there exists $C_{\varepsilon}>0$ such that
\begin{equation}\label{PSI}
|u|^{2}_{L^{2^{*}_{s}}(-\pi, \pi)^{N}}\leq (S_{*}^{-1}+\varepsilon) |(-\Delta)^{\frac{s}{2}}u|^{2}_{L^{2}(-\pi, \pi)^{N}}+C_{\varepsilon} |u|^{2}_{L^{2}(-\pi, \pi)^{N}} \quad \, \forall u\in \h, 
\end{equation}
and by using Theorem \ref{tracethm}, (\ref{A}) and (\ref{conv}) (strong convergence in $L^{2}(-\pi, \pi)^{N}$), we have
\begin{align}\begin{split}\label{yomo}
\|v_{j}-v\|_{\X}^{2}&= |\T(v_{j})-\T(v)|^{2^{*}_{s}}_{2^{*}_{s}, W}+o(1)\\
&\leq W(0) S_{*}^{-\frac{2^{*}_{s}}{2}}  [\T(v_{j}-v)]^{2^{*}_{s}} \\
&\leq W(0) S_{*}^{-\frac{2^{*}_{s}}{2}}\|v_{j}-v\|_{\X}^{2^{*}_{s}}.
\end{split}\end{align}
Therefore, if $\|v_{j}-v\|_{\X}^{2}\rightarrow l>0$, then from (\ref{D1}) and (\ref{yomo}) we deduce that
$$
W(0)^{-\frac{N-2s}{2s}} S_{*}^{\frac{N}{2s}}>l\geq W(0)^{-\frac{N-2s}{2s}} S_{*}^{\frac{N}{2s}},
$$
that is a contradiction.
Therefore, we can deduce that $v_{j} \rightarrow v$ strongly in $\X$.

\end{proof}

\begin{lem}\label{Linking3}
Let $M_{\e}:=\{y+tz_{\e}: y\in \Y, \|y+tz_{\e}\|_{\X}\leq R, t\geq 0\}$. Then, there exists $R>\rho$ sufficiently large such that
$\sup_{\partial M_{\e}} \mathcal{J}_{m}= 0$.
\end{lem}
\begin{proof}
Let $v=y+t z_{\e}\in \partial M_{\e}$. If $t=0$, it follows directly by (\ref{eqY}) and by the assumption $(f5)$ that $\mathcal{J}_{m}\leq 0 \mbox{ on } \Y$. Let $R= \|y+tz_{\e}\|_{\X}$ with $t>0$.
By using $(f4)$ and $(f5)$, we can see that for all $\delta>0$ there exists $C_{\delta}>0$ such that 
$$
F(x, t)\geq -\delta t^{2}+C_{\delta} |t|^{\beta} \mbox{ for all } t\in \R,
$$
with $\beta\in (2, 2^{*}_{s})$.
This gives
$$
\int_{\partial^{0}\mathcal{S}_{2\pi}} F(x, \T(v)) dx \geq -\delta|\T(y)+t\T(z_{\e})|_{2}^{2}+C_{\delta} |T(y)+t\T(z_{\e})|_{\beta}^{\beta}.
$$
On the other hand, by using H\"older inequality and Jensen's inequality, we can see that
\begin{align}\label{Val1}
|\T(y)+t\T(z_{\e})|_{2^{*}_{s}, W}^{2^{*}_{s}}&\geq C_{1}(|\T(y)+t\T(z_{\e})|^{2}_{2})^{\frac{2^{*}_{s}}{2}} \nonumber \\
&=C_{1} (|\T(y)|_{2}^{2}+t^{2}|\T(z_{\e})|^{2}_{2})^{\frac{2^{*}_{s}}{2}}
\end{align}
and
\begin{align}\label{Val2}
|\T(y)+t\T(z_{\e})|_{\beta}^{\beta}&\geq C_{2}(|\T(y)+t\T(z_{\e})|^{2}_{2})^{\frac{\beta}{2}} \nonumber \\
&=C _{2}(|\T(y)|_{2}^{2}+t^{2}|\T(z_{\e})|^{2}_{2})^{\frac{\beta}{2}}.
\end{align}
Then, \eqref{Val1} and \eqref{Val2} yield
\begin{align*}
\mathcal{J}_{m}(y+t z_{\e})&\leq \frac{1}{2} \|y+tz_{\e}\|^{2}_{\X}-\frac{m^{2s}}{2}|\T(y)+t\T(z_{\e})|_{2}^{2}-\frac{1}{2^{*}_{s}}|\T(y)+t\T(z_{\e})|_{2^{*}_{s}, W}^{2^{*}_{s}} \nonumber \\
&+\delta|\T(y)+t\T(z_{\e})|_{2}^{2}-C_{\delta} |T(y)+t\T(z_{\e})|_{\beta}^{\beta} \nonumber \\
&\leq \frac{t^{2}}{2} \left[\|z_{\e}\|^{2}_{\X}-\frac{m^{2s}}{2}|\T(z_{\e})|_{2}^{2}\right]-\frac{C_{1}}{2^{*}_{s}}(|\T(y)|_{2}^{2}+t^{2}|\T(z_{\e})|^{2}_{2})^{\frac{2^{*}_{s}}{2}} \nonumber \\
&+\delta (|\T(y)|_{2}^{2}+t^{2}|\T(z_{\e})|^{2}_{2})-C_{\delta} C_{2} (|\T(y)|_{2}^{2}+t^{2}|\T(z_{\e})|^{2}_{2})^{\frac{\beta}{2}} \nonumber \\
&\leq \frac{t^{2}}{2} \|z_{\e}\|^{2}_{\X}+\delta t^{2}|\T(z_{\e})|^{2}_{2}-C_{3} t^{2^{*}_{s}} |\T(z_{\e})|^{2^{*}_{s}}_{2}-C_{\delta} C_{4}t^{\beta} |\T(z_{\e})|^{\beta}_{2} \nonumber \\
&+\delta |\T(y)|^{2}_{2}-C_{3} |\T(y)|^{2^{*}_{s}}_{2}-C_{\delta}C_{4} |\T(y)|^{\beta}_{2}.
\end{align*}
In view of (\ref{important1}), we know that 
$$
\|z_{\e}\|^{2}_{\X}-m^{2s}|\T(z_{\e})|_{2}^{2}\leq S_{*}^{\frac{N}{2s}}+O(\e^{N-2s})
$$
so we deduce that 
\begin{equation}\label{normeq}
\|z_{\e}\|^{2}_{\X}\leq C_{5}+m^{2s}|\T(z_{\e})|_{2}^{2}\leq C_{5}+m^{2s} C_{6}.
\end{equation}
Therefore, we get
\begin{align}\label{belowm}
\mathcal{J}_{m}(y+t z_{\e})&\leq t^{2} (C_{5}+m^{2s}+\delta) |\T(z_{\e})|^{2}_{2}-C_{3} t^{2^{*}_{s}} |\T(z_{\e})|^{2^{*}_{s}}_{2}-C_{\delta} C_{4}t^{\beta} |\T(z_{\e})|^{\beta}_{2} \nonumber \\
&+\delta |\T(y)|^{2}_{2}-C_{3} |\T(y)|^{2^{*}_{s}}_{2}-C_{\delta}C_{4} |\T(y)|^{\beta}_{2}.
\end{align}
Taking into account $\|y+t z_{\e}\|_{\X}^{2}=m^{2s}|\T(y)|^{2}_{2}+t^{2}\|z_{\e}\|_{\X}^{2}$ and (\ref{normeq}), we can infer that when $\|y+t z_{\e}\|_{\X}\rightarrow \infty$ then $t\rightarrow \infty$ or $|\T(y)|_{2}\rightarrow \infty$, and this together with \eqref{belowm} yields
\begin{align*}
\mathcal{J}_{m}(y+t z_{\e})\rightarrow -\infty \mbox{ as } \|y+tz_{\e}\|_{\X}\rightarrow \infty.
\end{align*}

\end{proof}

Putting together Lemma \ref{Linking1}, Lemma \ref{Linking2}, Lemma \ref{psthm} and Lemma \ref{Linking3}, we can see that the assumptions of Theorem \ref{LinkingThm} are satisfied. Therefore, for all $m>0$ there exists $v_{m}\in \X$ such that $\mathcal{J}_{m}(v_{m})=c_{m}$ and $\mathcal{J}'_{m}(v_{m})=0$.
In particular, we know that $0<c_{m}<\frac{s}{N} W(0)^{-\frac{N-2s}{2s}} S_{*}^{\frac{N}{2s}}$.

\section{H\"older continuity of solutions of (\ref{P})}

\noindent
In this section we show that any solution of (\ref{P}) is a H\"older continuous function. 
\begin{lem}\label{lemmino}
Let $v\in \X$ be a weak solution to (\ref{R}). 
Then $\T(v)\in C^{0, \alpha}([-\pi,\pi]^{N})$, for some $\alpha \in (0,1)$.
\end{lem}
\begin{proof}
Since $v$ is a critical point for $\mathcal{J}_{m}$, we know that
\begin{align}\begin{split}\label{criticpoint}
\iint_{\mathcal{S}_{2\pi}} &\xi^{1-2s}(\nabla v \nabla \eta+m^{2}v \eta) \, dxd\xi\\
&=\int_{\partial^{0}\mathcal{S}_{2\pi}} [m^{2s}\T(v)+W(x)|\T(v)|^{2^{*}_{s}-2}\T(v)+ f(x,\T(v))]\T(\eta) \,dx
\end{split}\end{align}
for all $\eta\in \X$.\\
Let $w=vv^{2\beta}_{K}\in \X$ where $v_{K}=\min\{|v|,K\}$, $K>1$ and $\beta\geq 0$.
Taking $\eta=w$ in (\ref{criticpoint}), we deduce that 
\begin{align}\begin{split}\label{conto1}
\iint_{\mathcal{S}_{2\pi}} &\xi^{1-2s}v^{2\beta}_{K}(|\nabla v|^{2}+m^{2}v^{2}) \, dxd\xi+\iint_{D_{K}} 2\beta \xi^{1-2s}v^{2\beta}_{K} |\nabla v|^{2} \, dx d\xi  \\
&=\int_{\partial^{0}\mathcal{S}_{2\pi}} (m^{2s} \T(v)^{2}+W(x)|\T(v)|^{2^{*}_{s}}) \T(v_{K})^{2\beta} \,dx+ \int_{\partial^{0}\mathcal{S}_{2\pi}} f(x,\T(v))\T(v)\T(v_{K})^{2\beta} \,dx ,
\end{split}\end{align}
where $D_{K}=\{(x,\xi)\in \mathcal{S}_{2\pi}: |v(x, \xi)|\leq K\}$. \\
It is easy to see that
\begin{align}\begin{split}\label{conto2}
\iint_{\mathcal{S}_{2\pi}} &\xi^{1-2s}|\nabla (vv_{K}^{\beta})|^{2} dx\,d\xi  \\
&=\iint_{\mathcal{S}_{2\pi}} \xi^{1-2s}v_{K}^{2\beta} |\nabla v|^{2} dx\,d\xi+\iint_{D_{K}} (2\beta+\beta^{2}) \xi^{1-2s}v_{K}^{2\beta} |\nabla v|^{2} dx\,d\xi.
\end{split}\end{align}
Then, putting together (\ref{conto1}) and (\ref{conto2}) we get 
\begin{align}\begin{split}\label{S1}
&\|vv_{K}^{\beta}\|_{\X}^{2}\\
&=\iint_{\mathcal{S}_{2\pi}} \xi^{1-2s}[|\nabla (vv_{K}^{\beta})|^{2}+m^{2}v^{2}v_{K}^{2\beta}] dxd\xi  \\
&=\iint_{\mathcal{S}_{2\pi}} \xi^{1-2s}v_{K}^{2\beta}[ |\nabla v|^{2}+m^{2}v^{2}] dxd\xi+\iint_{D_{K}} 2\beta \Bigl(1+\frac{\beta}{2}\Bigr) \xi^{1-2s}v_{K}^{2\beta} |\nabla v|^{2} dxd\xi  \\
&\leq c_{\beta} \Bigl[\iint_{\mathcal{S}_{2\pi}} \xi^{1-2s}v_{K}^{2\beta}[ |\nabla v|^{2}+m^{2}v^{2}] dxd\xi+\iint_{D_{K}} 2\beta \xi^{1-2s}v_{K}^{2\beta} |\nabla v|^{2} dxd\xi\Bigr] \\
&=c_{\beta} \int_{\partial^{0}\mathcal{S}_{2\pi}} (m^{2s} \T(v)^{2}+W(x)|\T(v)|^{2^{*}_{s}}) \T(v_{K})^{2\beta} + f(x,\T(v))\T(v) \T(v_{K})^{2\beta} \,dx,
\end{split}\end{align}
where $c_{\beta}=1+\frac{\beta}{2}$.

By assumptions on $f$ and $W$, we deduce that  
\begin{align*}
(m^{2s}\T(v)^{2}&+W(x)|\T(v)|^{2^{*}_{s}})\T(v_{K})^{2\beta} + f(x,\T(v))\T(v) \T(v_{K})^{2\beta} \\
&\leq c_{1}(1+|\T(v)|^{2^{*}_{s}-2})\T(v)^{2}\T(v_{K})^{2\beta}+c_{2} |\T(v)|^{p-2} \T(v)^{2}\T(v_{K})^{2\beta} \mbox{ on } \partial^{0}\mathcal{S}_{2\pi}.
\end{align*}
Now, we prove that
\begin{equation*}
|\T(v)|^{2^{*}_{s}-2}+|\T(v)|^{p-2}\leq 1+h \mbox{ on } \partial^{0}\mathcal{S}_{2\pi},
\end{equation*}
for some $h\in L^{N/2s}(-\pi, \pi)^{N}$.
Firstly, we observe that
\begin{align*}
|\T(v)|^{p-2}&=\chi_{\{|\T(v)|\leq 1\}}|\T(v)|^{p-2}+\chi_{\{|\T(v)|>1\}}|\T(v)|^{p-2}\\
&\leq 1+\chi_{\{|\T(v)|>1\}} |\T(v)|^{p-2}\ \mbox{ on } \partial^{0}\mathcal{S}_{2\pi}.
\end{align*}
If $(p-2)N<4s$ then 
$$
\int_{\partial^{0}\mathcal{S}_{2\pi}} \chi_{\{|\T(v)|>1\}}|\T(v)|^{\frac{N}{2s}(p-2)} dx \leq \int_{\partial^{0}\mathcal{S}_{2\pi}} \chi_{\{|\T(v)|>1\}}|\T(v)|^{2} dx<\infty, 
$$
while if $4s\leq (p-2)N$ we have that $(p-2)\frac{N}{2s}\in [2,2^{*}_{s}]$.\\
Let us note that $|\T(v)|^{2^{*}_{s}-2}\in L^{\frac{N}{2s}}(-\pi, \pi)^{N}$.
Therefore, there exist a constant $C$ and a function $h\in L^{N/2s}(-\pi, \pi)^{N}$, $h\geq 0$ and independent of $K$ and $\beta$, such that
\begin{align}\begin{split}\label{S2}
(m^{2s}+|\T(v)|^{2^{*}_{s}-2})\T(v)^{2}\T(v)_{K}^{2\beta} &+ f(x,\T(v))\T(v)\T(v_{K})^{2\beta}\\
&\leq (C+h)v^{2}\T(v_{K})^{2\beta} \mbox{ on } \partial^{0}\mathcal{S}_{2\pi}.
\end{split}\end{align}
Taking into account (\ref{S1}) and (\ref{S2}) we have
\begin{equation*}
\|vv_{K}^{\beta}\|_{\X}^{2}\leq c_{\beta} \int_{\partial^{0}\mathcal{S}_{2\pi}} (C+h)\T(v)^{2}\T(v_{K})^{2\beta} dx,
\end{equation*}
and by the Monotone Convergence Theorem ($v_{K}$ is increasing with respect to $K$) we have as $K\rightarrow \infty$
\begin{equation}\label{i1}
\||v|^{\beta+1}\|_{\X}^{2}\leq Cc_{\beta} \int_{\partial^{0}\mathcal{S}_{2\pi}} |\T(v)|^{2(\beta +1)} dx + c_{\beta}\int_{\partial^{0}\mathcal{S}_{2\pi}}  h|\T(v)|^{2(\beta +1)}dx.
\end{equation}
Fix $M>0$ and let $A_{1}=\{h\leq M\}$ and $A_{2}=\{h>M\}$.

Then
\begin{equation}\label{i2}
\int_{\partial^{0}\mathcal{S}_{2\pi}}  h|\T(v)|^{2(\beta +1)} dx\leq M ||\T(v)|^{\beta+1}|_{2}^{2}+\varepsilon(M) ||\T(v)|^{\beta+1}|_{2^{*}_{s}}^{2},
\end{equation}
where $\displaystyle{\varepsilon(M)=\Bigl(\int_{A_{2}} h^{N/2s} dx \Bigr)^{\frac{2s}{N}}\rightarrow 0}$ as $M\rightarrow \infty$.
Taking into account (\ref{i1}) and (\ref{i2}), we get
\begin{equation}\label{regv}
\||v|^{\beta+1}\|_{\X}^{2}\leq c_{\beta}(c+M)||\T(v)|^{\beta+1}|_{2}^{2}+c_{\beta}\varepsilon(M)||\T(v)|^{\beta+1}|_{2^{*}_{s}}^{2}.
\end{equation}
By using Theorem $\ref{tracethm}$ we know that
\begin{equation}\label{S3}
||\T(v)|^{\beta+1}|_{2^{*}_{s}}^{2}\leq C^{2}_{2^{*}_{s}}\||v|^{\beta+1}\|_{\X}^{2}.
\end{equation}
Then, choosing $M$ large so that $\varepsilon(M) c_{\beta} C^{2}_{2^{*}}<\frac{1}{2}$,
and by using $(\ref{regv})$ and $(\ref{S3})$ we obtain
\begin{equation}\label{iter}
||\T(v)|^{\beta+1}|_{2^{*}_{s}}^{2}\leq 2 C^{2}_{2^{*}_{s}} \,c_{\beta}(c+M)||\T(v)|^{\beta+1}|^{2}_{2}.
\end{equation}
Then we can start a bootstrap argument: since $\T(v)\in L^{\frac{2N}{N-2s}}(-\pi, \pi)^{N}$ we can apply (\ref{iter}) with $\beta_{1}+1=\frac{N}{N-2s}$ to deduce that $\T(v)\in L^{\frac{(\beta_{1}+1)2N}{N-2s}}(-\pi, \pi)^{N}=L^{\frac{2N^{2}}{(N-2s)^{2}}}(-\pi, \pi)^{N}$. Applying (\ref{iter}) again, after $k$ iterations, we find $\T(v)\in L^{\frac{2N^{k}}{(N-2s)^{k}}}(-\pi, \pi)^{N}$, and so $\T(v)\in L^{q}(-\pi, \pi)^{N}$ for all $q\in[2,\infty)$. 
Then we can apply Proposition $3.5$ in \cite{FallFelli} to deduce that $\T(v)\in C^{0, \alpha}([-\pi, \pi]^{N})$, for some $\alpha\in (0, 1)$.

\end{proof}

\section{Periodic solutions for $m=0$}

\noindent
This last section is devoted to the proof of Theorem \ref{mthm2}. Firstly, we show that it is possible to estimate the critical levels from below and from above independently of $m$.\\
Fix $m_{0}\in (0, 1)$, and let us assume that $0<m< m_{0}$. Then we aim to prove that there exist $\sigma_{1}, \sigma_{2}>0$ independent of $m$, such that
\begin{equation}
\sigma_{1}\leq \mathcal{J}_{m}(v_{m})\leq \sigma_{2}
\end{equation}
for all $0<m< m_{0}$.
Firstly, we note that by using Proposition $2.1$ in \cite{benyi} and Theorem \ref{tracethm}, we can see that for all $z\in \Z$
$$
|\T(z)|^{2}_{2^{*}_{s}}\leq C_{*} [\T(z)]^{2}=C_{*} \sum_{k\in \mathbb{Z}^{N}} |k|^{2s}|c_{k}|^{2}\leq C_{*} |z|^{2}_{\h}\leq C_{*} \|z\|^{2}_{\X},
$$
where $c_{k}$ are the Fourier coefficients of $\T(z)$, and $C_{*}$ is independent of $m$. Then, by using H\"older inequality we can see that for all fixed $q\in [1, 2^{*}_{s}]$, there exists $C_{q}\equiv C(q, N, s)>0$ independent of $m$, such that 
$$
|\T(z)|_{q}\leq C_{q} \|z\|_{\X}
$$
for all $z\in \Z$. Let us note that when $q=2$, it results $C_{2}=1$.
Therefore, by using $(f3)$ and $(f4)$, we have for all $v\in \Z$ 
\begin{align*}
\mathcal{J}_{m}(v)&\geq \frac{1}{2} \|v\|_{\X}^{2}-\frac{m^{2s}}{2}|\T(v)|_{2}^{2}-\frac{W(0)}{2^{*}_{s}}|\T(v)|_{2^{*}_{s}}^{2^{*}_{s}} -\varepsilon |\T(v)|_{2}^{2}-C_{\varepsilon} |\T(v)|_{p}^{p} \\
&\geq \Bigl(\frac{1}{2}-\frac{m^{2s}_{0}}{2}-\varepsilon  \Bigr)\|v\|_{\X}^{2}-C'_{*}\|v\|^{2^{*}_{s}}_{\X}-C'_{p}C_{\e}\|v\|_{\X}^{p}.
\end{align*}
Choosing $\varepsilon$ sufficiently small, there exist $\sigma_{1}>0$ and $\eta>0$ such that
$$
\mathcal{J}_{m}(v)\geq \sigma_{1} \mbox{ for all  } v\in \Z: \|v\|_{\X}=\eta.
$$
Now, we can observe that, for any $v=y+tz_{\e}\in \partial M_{\e}$, we can replace the estimate (\ref{normeq}) and (\ref{belowm}) in Lemma \ref{Linking3} by 
\begin{equation}\label{normeqind}
\|z_{\e}\|^{2}_{\X}\leq C_{5}+m_{0}^{2s}|\T(z_{\e})|_{2}^{2},
\end{equation}
and 
\begin{align}\label{belowestimate}
\mathcal{J}_{m}(y+t z_{\e})&\leq t^{2} (C_{5}+m_{0}^{2s}+\delta) |\T(z_{\e})|^{2}_{2}-C_{3} t^{2^{*}_{s}} |\T(z_{\e})|^{2^{*}_{s}}_{2}-C_{\delta} C_{4}t^{\beta} |\T(z_{\e})|^{\beta}_{2} \nonumber \\
&+\delta |\T(y)|^{2}_{2}-C_{3} |\T(y)|^{2^{*}_{s}}_{2}-C_{\delta}C_{4} |\T(y)|^{\beta}_{2},
\end{align}
respectively. We also note that all constants appearing in (\ref{belowestimate}) are independent of $m$.

Since $\|y+t z_{\e}\|_{\X}^{2}=m^{2s}|\T(y)|^{2}_{2}+t^{2}\|z_{\e}\|_{\X}^{2}$ and by using (\ref{normeqind}), we can infer that 
\begin{align*}
\mathcal{J}_{m}(y+tz_{\e})\rightarrow -\infty \mbox{ as } \|y+tz_{\e}\|_{\X}\rightarrow \infty.
\end{align*}
In view of Lemma \ref{Linking2} and Lemma \ref{psthm}, we can find $0<\sigma_{2}< \frac{s}{N}W(0)^{-\frac{N-2s}{2s}}S_{*}^{\frac{N}{2s}}$ such that 
$\mathcal{J}_{m}(v_{m})=c_{m}\leq \sigma_{2}$, for all $m\in (0, m_{0})$.
By using $\mathcal{J}_{m}(v_{m})\leq \sigma_{2}$, $\mathcal{J}'_{m}(v_{m})=0$ and $(f5)$, we can deduce that 
\begin{align*}
\sigma_{2} &\geq \mathcal{J}_{m}(v_{m})-\frac{1}{2} \langle \mathcal{J}_{m}'(v_{m}), v_{m}\rangle \nonumber\\
&=\Bigl(\frac{1}{2}-\frac{1}{2^{*}_{s}}\Bigr) |\T(v_{m})|^{2^{*}_{s}}_{2^{*}_{s}, W} +\frac{1}{2}\int_{\partial^{0}\mathcal{S}_{2\pi}}  f(x, \T(v_{m})) \T(v_{m}) dx-\int_{\partial^{0}\mathcal{S}_{2\pi}}  F(x, \T(v_{m})) dx\nonumber \\
& \geq \Bigl(\frac{1}{2}-\frac{1}{2^{*}_{s}}\Bigr) |\T(v_{m})|^{2^{*}_{s}}_{2^{*}_{s}, W}+\left(\frac{1}{2}-\frac{1}{\mu}\right)\int_{\partial^{0}\mathcal{S}_{2\pi}}  f(x, \T(v_{m})) \T(v_{m}) dx \nonumber \\
&\geq \Bigl(\frac{1}{2}-\frac{1}{2^{*}_{s}}\Bigr) |\T(v_{m})|^{2^{*}_{s}}_{2^{*}_{s}, W}.
\end{align*}
Therefore $|\T(v_{m})|_{q}$ is bounded for all $q\in [1, 2^{*}_{s}]$. In particular, by using $(f2)$-$(f4)$, we have
\begin{align*}
\|v_{m}\|_{\X}^{2} &=2\mathcal{J}_{m}(v_{m})+m^{2s} |\T(v_{m})|^{2}_{2}+\frac{2}{2^{*}_{s}} |\T(v_{m})|^{2^{*}_{s}}_{2^{*}_{s}, W}+2 \int_{\partial^{0}\mathcal{S}_{2\pi}}  F(x, \T(v_{m})) dx \\
& \leq 2\sigma_{1} +m_{0}^{2s}|\T(v_{m})|^{2}_{2}+\frac{2}{2^{*}_{s}} |\T(v_{m})|^{2^{*}_{s}}_{2^{*}_{s}, W}+C_{1} |\T(v_{m})|_{2}^{2}+C_{2} |\T(v_{m})| ^{p}_{p}\leq C_{3}+C_{4},
\end{align*}
for all $m\in (0, m_{0})$.
This means that $(v_{m})_{m}$ is bounded in ${\mathbb{X}^{s}_{m}}$. 
In particular, we can see that
\begin{equation}\label{1v}
C_{3}+C_{4}\geq \|v_{m}\|^{2}_{\X}\geq \|\nabla v_{m}\|^{2}_{L^{2}(\mathcal{S}_{2\pi},\xi^{1-2s})}
\end{equation}
and
\begin{align}\label{2v}
C_{3}+C_{4}\geq \|v_{m}\|^{2}_{\X}&\geq |\T(v_{m})|^{2}_{\h}\geq [\T(v_{m})]^{2}.
\end{align} 
Now, we prove that for any $\delta>0$, the following inequality holds true
\begin{align}\label{nash}
\|v\|_{L^{2}((-\pi,\pi)^{N}\times (0,\delta),\xi^{1-2s})}^2 &\leq \frac{\delta^{2-2s}}{1-s} |\T(v)|_{2}^{2}+\frac{\delta ^{2}}{2s} \|\partial_{\xi} v\|_{L^{2}(\mathcal{S}_{2\pi},\xi^{1-2s})}^{2}
\end{align}
for any $v\in \X$.
Fix $\delta>0$ and let $v\in C^{\infty}_{T}(\overline{\R^{N+1}_{+}})$ be such that $\|v\|_{\X}<\infty$.
For any $x\in [0,T]^{N}$ and $\xi\in [0, \delta]$, we have
$$
v(x,\xi)=v(x,0)+\int_{0}^{\xi} \partial_{\xi} v(x,t) dt.
$$
By using $(a+b)^{2}\leq 2a^{2}+2b^{2}$ for all $a, b\geq 0$ we obtain
$$
|v(x,\xi)|^2 \leq 2  |v(x,0)|^{2}+2\Bigl(\int_{0}^{\xi}|\partial_{\xi} v(x,t)| dt\Bigr)^{2},
$$
and applying the H\"older inequality we deduce
\begin{equation}\label{vtii5}
|v(x,\xi)|^2 \leq 2 \Bigl[ |v(x,0)|^{2}+\Bigl(\int_{0}^{\xi} t^{1-2s}|\partial_{\xi} v(x,t)|^{2}dt\Bigr)\frac{\xi^{2s}}{2s}\,  \Bigr].
\end{equation}
Multiplying both members of (\ref{vtii5}) by $\xi^{1-2s}$ we get
\begin{equation}\label{vtii}
\xi^{1-2s}|v(x,\xi)|^2 \leq 2 \Bigl[ \xi^{1-2s}|v(x,0)|^{2}+\Bigl(\int_{0}^{\xi} t^{1-2s} |\partial_{\xi} v(x,t)|^{2}dt\Bigr)\frac{\xi}{2s} \Bigr].
\end{equation}
Integrating (\ref{vtii}) over $(-\pi,\pi)^{N}\times (0,\delta)$ we have
\begin{align}\label{nash}
\|v\|_{L^{2}((-\pi,\pi)^{N}\times (0,\delta),\xi^{1-2s})}^2 &\leq \frac{\delta^{2-2s}}{1-s} |v(\cdot,0)|_{L^{2}(-\pi,\pi)^{N}}^{2}+\frac{\delta ^{2}}{2s} \|\partial_{\xi} v\|_{L^{2}(\mathcal{S}_{2\pi},\xi^{1-2s})}^{2}.
\end{align}
By density we get the desired result.\\
Then, by Theorem \ref{thm2}, we can extract a subsequence, which we denote again by $v_{m}$, and a function $v$ satisfying $v\in L^{2}_{loc}(\mathcal{S}_{2\pi},\xi^{1-2s})$, $\nabla v\in L^{2}(\mathcal{S}_{2\pi},\xi^{1-2s})$, such that as $m\rightarrow 0$ we have
\begin{align}\begin{split}\label{convergenzafinale}
&v_{m}\rightharpoonup v \mbox{ in } L^{2}_{loc}(\mathcal{S}_{2\pi},\xi^{1-2s}), \\
&\nabla v_{m}\rightharpoonup \nabla v \mbox{ in } L^{2}(\mathcal{S}_{2\pi},\xi^{1-2s}), \\ 
&\T(v_{m})  \rightarrow \T(v)   \mbox{ in } L^{q}(-\pi, \pi)^{N} \mbox{ for } q\in [1, 2^{*}_{s}), \\
&\T(v_{m})   \rightarrow \T(v)   \mbox{ a.e. in }  (-\pi, \pi)^{N}. 
\end{split}\end{align}
At this point, we prove that $v$ is a weak solution to 
\begin{equation}\label{R'}
\left\{
\begin{array}{ll}
-\dive(\xi^{1-2s} \nabla v) =0 &\mbox{ in }\mathcal{S}_{2\pi} \\
v_{| {\{x_{i}=0\}}}= v_{| {\{x_{i}=T\}}} & \mbox{ on } \partial_{L}\mathcal{S}_{2\pi} \\
\frac{\partial v}{\partial \nu^{1-2s}}=W(x)|v|^{2^{*}_{s}-2} v+ f(x,v)   &\mbox{ on }\partial^{0}\mathcal{S}_{2\pi}.
\end{array}
\right.
\end{equation}
We know that $v_{m}$ satisfies 
\begin{align}\begin{split}\label{spqr}
\iint_{\mathcal{S}_{2\pi}} &\xi^{1-2s}(\nabla v_{m} \nabla \eta+m^{2}v_{m}\eta) \; dx\,d\xi\\
&=\int_{\partial^{0}\mathcal{S}_{2\pi}} [m^{2s}\T(v_{m})+W(x)|\T(v_{m})|^{2^{*}_{2}-2} \T(v_{m})+f(x,\T(v_{m}))] \T(\eta)  \; dx
\end{split}\end{align}
for every $\eta \in \X$.
Now, fix $\varphi \in C^{\infty}_{2\pi}(\overline{\R^{N+1}_{+}})$ such that $\nabla \varphi\in L^{2}(\mathcal{S}_{2\pi},\xi^{1-2s})$, and we introduce $\psi\in \mathcal{C}^{\infty}([0,\infty))$ defined as follows
\begin{equation}\label{xidef}
\left\{
\begin{array}{cc}
\psi=1 &\mbox{ if } 0\leq \xi\leq 1 \\
0\leq \psi \leq 1 &\mbox{ if } 1\leq \xi\leq 2 \\ 
\psi=0 &\mbox{ if } \xi\geq 2. 
\end{array}
\right.
\end{equation}
We set $\psi_{R}(\xi):=\psi(\frac{\xi}{R})$ for $R>1$. Then choosing $\eta=\varphi \psi_{R}\in \X$ in (\ref{spqr}) and taking the limit as $m\rightarrow 0$ we have 
\begin{equation}\label{limitR}
\iint_{\mathcal{S}_{2\pi}} \xi^{1-2s}\nabla v \nabla (\varphi \psi_{R}) \; dxd\xi=\int_{\partial^{0}\mathcal{S}_{2\pi}} \left[W(x)|\T(v)|^{2^{*}_{2}-2} \T(v)+ f(x,\T(v))\right]\T(\varphi) \; dx.
\end{equation}
By passing to the limit in (\ref{limitR}) as $R\rightarrow \infty$, we deduce that $v$ verifies
$$
\iint_{\mathcal{S}_{2\pi}} \xi^{1-2s} \nabla v \nabla \varphi \; dxd\xi=\int_{\partial^{0}\mathcal{S}_{2\pi}} \left[W(x)|\T(v)|^{2^{*}_{2}-2} \T(v)+ f(x,\T(v))\right]\T(\varphi) \; dx
$$
for any $\varphi \in C^{\infty}_{2\pi}(\overline{\R^{N+1}_{+}})$ such that $\nabla \varphi\in L^{2}(\mathcal{S}_{2\pi},\xi^{1-2s})$, so by density for all $\varphi\in \dot{\mathbb{X}}^{s}$.

Finally we show that $v$ is not identically zero. 
Let us denote by $\mathcal{J}_{0}$ the Euler-Lagrange functional associated to (\ref{R'}), that is
$$
\mathcal{J}_{0}(v)=\frac{1}{2}\|\nabla v\|_{L^{2}(\mathcal{S}_{2\pi}, \xi^{1-2s})}^{2}-\frac{1}{2^{*}_{s}}|\T(v)|_{2^{*}_{s}, W}^{2^{*}_{s}}- \int_{\partial^{0}\mathcal{S}_{2\pi}} F(x, \T(v)) dx
$$
for all $v\in \dot{\mathbb{X}}^{s}$.
Now, we proceed as in the proof of Lemma \ref{psthm}.
By using (\ref{convergenzafinale}) and $(f2)$-$(f4)$, we can see that
\begin{equation}\label{pf11}
\int_{\partial^{0}\mathcal{S}_{2\pi}}  F(x, \T(v_{m})) dx\rightarrow \int_{\partial^{0}\mathcal{S}_{2\pi}}  F(x, \T(v)) dx,
\end{equation}
\begin{equation}\label{pf22}
\int_{\partial^{0}\mathcal{S}_{2\pi}}  f(x, \T(v_{m})) (\T(v_{m})-\T(v)) dx \rightarrow 0,
\end{equation}
and
\begin{equation}\label{pf33}
\int_{\partial^{0}\mathcal{S}_{2\pi}}  f(x, \T(v)) (\T(v_{m})-\T(v)) dx \rightarrow 0
\end{equation}
as $m\rightarrow 0$.
Since $\langle \mathcal{J}_{0}'(v), v \rangle=0$, we have
\begin{align*}
0=\|\nabla v\|_{L^{2}(\mathcal{S}_{2\pi}, \xi^{1-2s})}^{2} - |\T(v)|^{2^{*}_{s}}_{2^{*}_{s}, W} -\int_{\partial^{0}\mathcal{S}_{2\pi}}  f(x, \T(v)) \T(v) dx
\end{align*}
and by using $(f5)$ we get 
\begin{equation}
\mathcal{J}_{0}(v)\geq \Bigl(\frac{1}{2}-\frac{1}{2^{*}_{s}}\Bigr) |\T(v)|^{2^{*}_{s}}_{2^{*}_{s}, W} +\left(\frac{1}{2}-\frac{1}{\mu}\right)\int_{\partial^{0}\mathcal{S}_{2\pi}}  f(x, \T(v)) \T(v) dx\geq 0. 
\end{equation}
By Brezis-Lieb Lemma \cite{BL}, we can note that as $m\rightarrow 0$
\begin{align}\label{a22}
\|\nabla v_{m}\|_{L^{2}(\mathcal{S}_{2\pi}, \xi^{1-2s})}^{2}  = \|\nabla (v_{m}-v)\|_{L^{2}(\mathcal{S}_{2\pi}, \xi^{1-2s})}^{2} + \|\nabla v\|_{L^{2}(\mathcal{S}_{2\pi}, \xi^{1-2s})}^{2} +o(1)
\end{align}
and
\begin{align}\label{a11}
|\T(v_{m})|^{2^{*}_{s}}_{2^{*}_{s},W} \,dx=|\T(v_{m})-\T(v)|^{2^{*}_{s}}_{2^{*}_{s}, W} + |\T(v)|^{2^{*}_{s}}_{2^{*}_{s}, W}+o(1).
\end{align}
Thus, by using (\ref{pf11}), (\ref{a22}), (\ref{a11}) and the fact that $\T(v_{m}) \rightarrow \T(v)$ in $L^{2}(-\pi,\pi)^{N}$, we have 
\begin{align}\label{C1}
\mathcal{J}_{0}(v_{m})=\mathcal{J}_{0}(v)+\left[\frac{1}{2}\|\nabla(v_{m}-v)\|_{L^{2}(\mathcal{S}_{2\pi}, \xi^{1-2s})}^{2}-\frac{1}{2^{*}_{s}} |\T(v_{m})-\T(v)|^{2^{*}_{s}}_{2^{*}_{s}, W}\right]+o(1). 
\end{align}
Since $\T(v_{m}) \rightharpoonup \T(v)$ in $L^{2^{*}_{s}}(-\pi, \pi)^{N}$, we can see
\begin{align*}
\int_{\partial^{0}\mathcal{S}_{2\pi}} &W(x)(|\T(v_{m})|^{2^{*}_{s}-2} \T(v_{m})-|\T(v)|^{2^{*}_{s}-2} \T(v))(\T(v_{m})-\T(v)) \,dx  \nonumber \\
&=\int_{\partial^{0}\mathcal{S}_{2\pi}} W(x) (|\T(v_{m})|^{2^{*}_{s}}-|\T(v_{m})|^{2^{*}_{s}-2} \T(v_{m}) \T(v)) \,dx+o(1)  \nonumber \\
&=\int_{\partial^{0}\mathcal{S}_{2\pi}} W(x) (|\T(v_{m})|^{2^{*}_{s}}-|\T(v)|^{2^{*}_{s}})  \,dx+o(1) \nonumber \\
&=\int_{\partial^{0}\mathcal{S}_{2\pi}} W(x) |\T(v_{m})-\T(v)|^{2^{*}_{s}} \,dx+o(1).
\end{align*}
This, (\ref{pf22}), (\ref{pf33}) and $\langle \mathcal{J}'_{0}(v), v_{m}-v\rangle=0$ (we recall that $\{\nabla v_{m}\}$ and $\{\T(v_{m})\}$ are bounded in $L^{2}(\mathcal{S}_{2\pi}, \xi^{1-2s})$ and $L^{2^{*}_{s}}(-\pi, \pi)^{N}$ respectively, so $v_{m}-v$ can be used as test function) yield
\begin{align*}
0&=\langle \mathcal{J}_{0}'(v_{m}), v_{m}-v \rangle \\
&=\langle \mathcal{J}_{0}'(v_{m})-\mathcal{J}_{0}'(v), v_{m}-v \rangle \\
&= \iint_{\mathcal{S}_{2\pi}} \xi^{1-2s}|\nabla (v_{m}- v)|^{2} \,dxd\xi-\int_{\partial^{0}\mathcal{S}_{2\pi}} W(x) |\T(v_{m})-\T(v)|^{2^{*}_{s}} \,dx+o(1),
\end{align*}
that is 
\begin{equation}\label{AG}
\|\nabla (v_{m}-v)\|_{L^{2}(\mathcal{S}_{2\pi}, \xi^{1-2s})}^{2}= |\T(v_{m})-\T(v)|^{2^{*}_{s}}_{2^{*}_{s}, W}+o(1).
\end{equation}
Taking into account $\mathcal{J}_{0}(v) \geq 0$, $\T(v_{m})\rightarrow \T(v)$ as $m\rightarrow 0$, and (\ref{C1}), we can infer that as $m\rightarrow 0$
\begin{align*}
\frac{1}{2}\|\nabla (v_{m}-v)\|_{L^{2}(\mathcal{S}_{2\pi}, \xi^{1-2s})}^{2}&-\frac{1}{2^{*}_{s}} |\T(v_{m})-\T(v)|^{2^{*}_{s}}_{2^{*}_{s}, W} \nonumber\\
&=\mathcal{J}_{0}(v_{m})-\mathcal{J}_{0}(v)+o(1) \nonumber \\
&\leq \mathcal{J}_{0}(v_{m})+o(1) \nonumber \\
&\leq \mathcal{J}_{m}(v_{m})+o(1) \nonumber \\
&\leq \sigma_{2}+o(1)<\frac{s}{N} W(0)^{-\frac{N-2s}{2s}} S_{*}^{\frac{N}{2s}},
\end{align*}
and by using (\ref{AG}) and $\frac{1}{2}-\frac{1}{2^{*}_{s}}=\frac{s}{N}$, we obtain for $m$ sufficiently small
\begin{equation}\label{D1G}
\|\nabla (v_{m}-v)\|_{L^{2}(\mathcal{S}_{2\pi}, \xi^{1-2s})}^{2}<W(0)^{-\frac{N-2s}{2s}} S_{*}^{N/2s}.
\end{equation}
Now, from the property $(E2)$ of the extension and the trace inequality with $m=0$ (in this case one has to replace $\X$ by $\dot{\mathbb{X}}^{s}$ and $\h$ by $\dot{\mathbb{H}}^{s}$ in Theorem \ref{tracethm}), 
we note that for any $u\in \dot{\mathbb{X}}^{s}$, it holds
\begin{equation}\label{AHPno}
[\T(u)]^{2}=|(-\Delta)^{\frac{s}{2}}\T(u)|^{2}_{L^{2}(-\pi, \pi)^{N}}\leq \|\nabla u\|_{L^{2}(\mathcal{S}_{2\pi}, \xi^{1-2s})}^{2}. 
\end{equation}
Therefore, thanks to (\ref{AG}), (\ref{yomo}) and (\ref{AHPno}), we have
\begin{align*}
\|\nabla (v_{m}-v)\|_{L^{2}(\mathcal{S}_{2\pi}, \xi^{1-2s})}^{2}&= |\T(v_{m})-\T(v)|^{2^{*}_{s}}_{2^{*}_{s}, W}+o(1) \nonumber \\
&\leq W(0) S_{*}^{-\frac{2^{*}_{s}}{2}} [\T(v_{m})-\T(v)]^{2^{*}_{s}} \nonumber\\
&\leq W(0) S_{*}^{-\frac{2^{*}_{s}}{2}} \|\nabla (v_{m}-v)\|_{L^{2}(\mathcal{S}_{2\pi}, \xi^{1-2s})}^{2^{*}_{s}}.
\end{align*}
Then, in view of (\ref{D1G}), we can deduce that $\|\nabla (v_{m}-v)\|_{L^{2}(\mathcal{S}_{2\pi}, \xi^{1-2s})}\rightarrow 0$ as $m\rightarrow 0$. Moreover, by (\ref{AG}), we get $|\T(v_{m})-\T(v)|_{2^{*}_{s}}\rightarrow 0$ as $m\rightarrow 0$.

Hence, putting together $\mathcal{J}_{m}(v_{m})\geq \sigma_{1}>0$, $\langle \mathcal{J}'_{m}(v_{m}), v_{m}\rangle=0$ and the growth assumptions on $f$, we can see that 
\begin{align*}
\sigma_{1}&\leq\mathcal{J}_{m}(v_{m})-\frac{1}{2} \langle \mathcal{J}'_{m}(v_{m}), v_{m} \rangle  \\
&=\left(\frac{1}{2}-\frac{1}{2^{*}_{s}}\right)|\T(v_{m})|_{2^{*}_{s}, W}^{2^{*}_{s}}+ \int_{\partial^{0}\mathcal{S}_{2\pi}} \left[\frac{1}{2} f(x, \T(v_{m})) \T(v_{m})-F(x, \T(v_{m}))\right] dx \\
&\leq c_{1}|\T(v_{m})|_{2^{*}_{s}}^{2^{*}_{s}}+c_{2} |\T(v_{m})|_{2}^{2}+c_{3}|\T(v_{m})|_{p}^{p},
\end{align*}
and taking the limit as $m\rightarrow 0$ in this inequality (now we know that $\T(v_{m})\rightarrow \T(v)$ in $L^{q}(-\pi, \pi)^{N}$ for any $q\in [2, 2^{*}_{s}]$), we deduce that $\T(v)$ is not identically zero. Moreover, by using the weak formulation of (\ref{P'}), $(f5)$ and $(W1)$, we can infer that $\T(v)$ cannot be constant.


\smallskip

\noindent
{\bf Acknowledgements.} 
The author warmly thanks the anonymous referee for her/his useful and nice comments on the paper. 
The manuscript has been carried out under the auspices of the INDAM - Gnampa Project 2017 titled:{\it Teoria e modelli per problemi non locali}.

\end{document}